\documentclass[12pt,reqno]{amsart}
\advance \topmargin by -\headheight
\advance \topmargin by -\headsep
\evensidemargin \oddsidemargin
\usepackage{amsmath}
\usepackage{amsfonts}
\usepackage{stmaryrd}
\usepackage{amssymb}
\usepackage{amsthm}
\usepackage{csquotes}

\usepackage{mathrsfs}
\usepackage{dsfont}
\usepackage{bm}
\usepackage{cite}
\usepackage{soul}

\numberwithin{equation}{section}
\renewcommand\vec{\bm}
\newcommand{\n}[1]{\|{#1}\|}

\newtheorem{theorem}{Theorem}[section]

\newtheorem{lemma}[theorem]{Lemma}
\newtheorem{Proposition}[theorem]{Proposition}
\newtheorem{Conjecture}[theorem]{Conjecture}
\newtheorem{Corollary}[theorem]{Corollary}

\usepackage{mathtools}

\title[Additive energies on spheres]{Additive energies on spheres}
\author[Akshat Mudgal]{Akshat Mudgal}
\date{} 
\address{Mathematical Institute, University of Oxford, Radcliffe Observatory Quarter, Woodstock Road, Oxford OX2 6GG, UK}
\email{Akshat.Mudgal@maths.ox.ac.uk}

\renewcommand\vec{\bm}

\begin{document}

\maketitle

\begin{abstract}
In this paper, we study additive properties of finite sets of lattice points on spheres in $3$ and $4$ dimensions. Thus, given $d,m \in \mathbb{N}$, let $A$ be a set of lattice points $(x_1, \dots, x_d) \in \mathbb{Z}^d$ satisfying $x_1^2 + \dots + x_{d}^2 = m$. When $d=4$, we prove threshold breaking bounds for the additive energy of $A$, that is, we show that there are at most $O_{\epsilon}(m^{\epsilon}|A|^{2 + 1/3 - 1/2766})$ solutions to the equation $a_1 + a_2 = a_3 + a_4,$ with $a_1, \dots, a_4 \in A$. This improves upon a result of Bourgain and Demeter, and makes progress towards one of their conjectures. A further novelty of our method is that we are able to distinguish between the case of the sphere and the paraboloid in $\mathbb{Z}^4$, since the threshold bound is sharp in the latter case. We also obtain variants of this estimate when $d=3$, where we improve upon previous results of Benatar and Maffucci concerning lattice point correlations. Finally, we use our bounds on additive energies to deliver discrete restriction type estimates for the sphere.
\end{abstract}

\section{Introduction}

A classically studied object in additive number theory, harmonic analysis and incidence geometry is the $(d-1)$-sphere, that is, the unit sphere in $\mathbb{R}^d$. Workers in the former areas are often interested in finding additive properties of points on the $(d-1)$-sphere, including topics like the behaviour of solutions to a given additive equation where all the variables lie in some finite subset of the $(d-1)$-sphere. Moreover, these relate, in a natural way, to moments of exponential sums supported on finite subsets of the sphere, thus highlighting the connections of this topic to restriction theory on curved surfaces. On the other hand, incidence geometry involves studying bounds on the number of incidences between an arbitrary finite set of points and an arbitrary set of varieties, which in this particular case, would entail a finite collection of spheres. In recent times, this connection has been frequently exploited by applying results of an incidence geometric flavour to improve estimates that are central to number theory and harmonic analysis (for instance, see \cite{BM2019}, \cite{BB2015}, \cite{BD2015a}).
\par

In this paper, we bring techniques from arithmetic combinatorics into this blend, consequently strengthening various known results on additive properties associated with subsets of $(d-1)$-spheres, when $d \in \{3,4\}$. Hence, given a natural number $d$ and real number $\lambda>0$, we define
\[ \mathcal{S}_{d, \lambda} = \{ \vec{x} \in \mathbb{R}^d \ | \ x_1^2 + \dots + x_d^2 = \lambda\}. \]
Moreover, for natural numbers $m$, we use $S_{d,m}$ to denote the lattice points on the sphere $\mathcal{S}_{d,m}$, that is, $S_{d,m} = \mathcal{S}_{d,m} \cap \mathbb{Z}^d$. 
Since we are interested in studying additive equations over the sphere, we define for every $s \in \mathbb{N}$ and every finite, non-empty subset $A$ of $\mathcal{S}_{d,\lambda}$, the additive energy $E_{s,2}(A)$, which counts the number of solutions to the equation
\begin{equation} \label{abad}
 \vec{x}_1 + \dots + \vec{x}_{s} = \vec{x}_{s+1} + \dots + \vec{x}_{2s}, 
\end{equation}
such that $\vec{x}_1, \dots, \vec{x}_{2s} \in A$. This quantity $E_{s,2}(A)$ has been studied in various works (see \cite{De2014}, \cite{BD2015}), and estimates involving this have close connections to problems of a harmonic analytic flavour such as discrete restriction estimates for spheres.
\par

We begin by studying bounds for $E_{2,2}(A)$ for subsets $A$ of $S_{4,m}$. The best known result in this direction was given by Bourgain and Demeter \cite{BD2015a} who showed that
\begin{equation} \label{bdbnd}
 E_{2,2}(A) \ll_{\epsilon} m^{\epsilon} |A|^{2 + 1/3}. 
 \end{equation}
In particular, they used ideas from incidence geometry to obtain the above bound, and further noted that these techniques worked in a similar manner for lattice points on paraboloids, consequently delivering the same bound as $\eqref{bdbnd}$ even in the latter scenario. But in the case of paraboloids, this bound is sharp, and so, they remarked that any further progress in the spherical case should require some involved number theory to detect the non-uniform distribution of lattice points on the sphere. Furthermore, they stated that, in the spherical setting, a much stronger bound than $\eqref{bdbnd}$ should hold true. In fact, this conjectured estimate is equivalent to the four dimensional case of a well known problem in harmonic analysis known as the discrete restriction conjecture for lattice points on the sphere.

\begin{Conjecture}\label{bdconsp}
Let $A$ be a non-empty subset of $S_{4,m}$. Then, we have
\[  E_{2,2}(A) \ll_{\epsilon} m^{\epsilon} |A|^{2}. \]
 \end{Conjecture}
 In particular, this would imply that for every $s \geq 2$ and for every non-empty subset $A$ of $S_{4,m}$, we have $E_{s,2}(A) \ll_{\epsilon} m^{\epsilon} |A|^{2s-2}.$ We now state our main result which provides threshold breaking upper bounds for $E_{s,2}(A)$ in this setting. 

\begin{theorem} \label{btdec}
Let $s,m$ be natural numbers such that $s \geq 2$, let $A$ be a finite, non-empty subset of $S_{4,m}$ and let $c = 1/461$. Then
\[ E_{s,2}(A) \ll_{s,\epsilon} m^{\epsilon} |A|^{2s- 2 + 1/6 + (1- c)\cdot 6^{-s+1}}. \]
\end{theorem}

When $s=2$, this gives us the bound
\[ E_{2,2}(A) \ll_{\epsilon} m^{\epsilon} |A|^{2 + 1/3 - 1/2766},\]
which obtains a power saving over the above-mentioned threshold bound in $\eqref{bdbnd}$, thus improving upon the work of Bourgain and Demeter and making progress towards their conjecture. Moreover, noting Conjecture $\ref{bdconsp}$ and the discussion thereafter, we see that Theorem $\ref{btdec}$ misses the conjectured bounds by a factor of at most $|A|^{1/6 + (1- c)\cdot 6^{-s+1}}$.
\par

We utilise a variety of different incidence geometric results in our proof of Theorem $\ref{btdec}$, and one of the crucial ingredients antecedent to these methods is the fact that any three distinct translates of $S_{4,m}$ intersect in at most $O_{\epsilon}(m^{\epsilon})$ points of $\mathbb{Z}^4$. Moreover, this does not hold true when $S_{4,m}$ is replaced by the set $P_{4,m}$ of lattice points on the truncated paraboloid, that is,
\begin{equation} \label{pardef}
 P_{4,m} = \{ (n_1, n_2, n_3, n_1^2 + n_2^2 + n_3^2) \in \mathbb{Z}^4 \ | \ -m \leq n_1, n_2, n_3 \leq m\}. 
 \end{equation} 
We employ this fact in conjunction with results concerning point--hyperplane and point--sphere incidences to obtain threshold bounds for $E_{s,2}(A)$ whenever $s \geq 2$. We will then use these bounds together with the higher energy method from arithmetic combinatorics as well as various elementary combinatorial geometric arguments to obtain threshold breaking bounds for $E_{2,2}(A)$, which, in turn, combine with the aforementioned circle of ideas to furnish threshold breaking bounds for $E_{s,2}(A)$ whenever $s \geq 2$. Moreover, these bounds on additive energies are closely related to discrete restriction estimates, and in fact, our main result delivers results of the latter flavour as well, see, for instance, Corollaries $\ref{restrict}$ and $\ref{restrict2}$. We point the reader to \S2 for more details regarding this, as well as for the discussion on discrepancies between the sphere and the paraboloid in the lattice point setting.
\par

We note that our methods also supply upper bounds for $E_{s,2}(A)$ when $A$ is a finite set of lattice points on the sphere in three dimensions.

\begin{theorem} \label{main}
Let $s \geq 2$ and let $A$ be a non-empty subset of $S_{3,m}$. Then, writing $\eta_{s} = 2^{-1} \cdot 3^{-s+2}$, we have
\[ E_{s,2}(A) \ll_{s, \epsilon} m^{\epsilon} |A|^{2s -3 + 1/2 + \eta_{s}}. \]
\end{theorem}

As before, not only do the above bounds give rise to various discrete restriction type results, but they can be further used to estimate other measures of additivity for a subset $A$ of $S_{3,m}$. For instance, we see that additive energy is closely related to the number of ways we could represent an arbitrary element of $\mathbb{R}^3$ as an $s$-fold sum of elements of $A$. Thus, writing $r_{s}(A,\vec{n})$ to count the number of solutions to 
\[ \vec{x}_1 + \dots + \vec{x}_{s} = \vec{n}\]
with $\vec{x}_i \in A$ for each $1 \leq i \leq s$, we see that
\begin{equation} \label{cs1} 
E_{s,2}(A) = \sum_{\vec{n} \in sA} r_{s}(A,\vec{n})^2, 
\end{equation}
where $sA$ is the set of all $s$-fold sums from $A$, that is, 
\[ sA = \{ \vec{a}_1 + \dots + \vec{a}_s \ | \ \vec{a}_1, \dots, \vec{a}_s \in A \}. \]
Note that this is precisely the set of elements $\vec{n}$ for which $r_{s}(A,\vec{n}) >0$. Using our methods, we also obtain non-trivial estimates for 
\[ r_{s}(A) = \sup_{\vec{n}} r_{s}(A,\vec{n}), \]
that is, the $l_{\infty}$ norm of the representation function $r_{s}(A, \cdot)$.

\begin{theorem} \label{linf}
Let $s \geq 4$, let $m \in \mathbb{N}$, and let $A \subseteq S_{3,m}$. Then we have
\[ r_{s}(A) \ll_{s,\epsilon} m^{\epsilon} |A|^{s-3 + 1/2 + \lambda_{s} }, \]
where $\lambda_{s} =  2^{-1} 3^{-s/2 + 2}$ when $s$ is even, and $\lambda_{s} =  10^{-1} \cdot 3^{-(s-1)/2+3}$ when $s$ is odd.
\end{theorem}

In the specific case when $A = S_{3,m}$, Benatar and Maffucci \cite{BM2019} called $r_{s}(S_{3,m},0)$ the number of $3$-dimensional $s$-th lattice point correlations, and showed that upper bounds for these objects contribute to understanding the variance of the so called nodal area of random Gaussian Laplace eigenfunctions on $\mathbb{T}^3$. In this endeavour, they proved that whenever $s \geq 6$, one has
\begin{equation} \label{bmbd}
 r_{s}(S_{3,m},0) \ll_{\epsilon}  m^{\epsilon}|S_{3,m}|^{s- 3 + 2/3},  
 \end{equation}
and further asked about sharp estimates for $r_{s}(S_{3,m},0)$ (see \cite[Question $1.5$]{BM2019}). Thus, Theorem $\ref{linf}$ makes progress towards their question by providing bounds stronger than $\eqref{bmbd}$ whenever $s \geq 7$. Likewise, we also produce some improvement for small values of $s$. We note that Benatar and Maffucci \cite{BM2019} proved that
\[ r_{5}(S_{3,m},0) \ll_{\epsilon} m^{\epsilon}|S_{3,m}|^{2 + 5/6}, \]
while Theorem $\ref{linf}$ implies that
\[ r_{5}(S_{3,m},0) \leq r_{5}(S_{3,m}) \ll_{\epsilon} m^{\epsilon} |S_{3,m}|^{2 + 4/5}, \] 
thus, improving upon the previous known bound. 
\par

In a similar vein, we also note that when $s=2p+1$, for some $p \geq 2$, a standard application of the Cauchy-Schwarz inequality combined with our estimates from Theorem $\ref{main}$ gives us
\[ r_{s}(A,\vec{n}) \leq E_{p,2}(A)^{1/2} E_{p+1,2}(A)^{1/2}   \ll_{s,\epsilon} |A|^{s-3 + 1/2 + 2^{-1}(\eta_{p} + \eta_{p+1}) }, \]
for any $\vec{n}$. While these are already non-trivial bounds, we note that
\[ 2^{-1} (\eta_{p} + \eta_{p+1}) = 3^{-p+1} > 10^{-1} \cdot 3^{-p+3} = \lambda_{s},\]
thus indicating that Theorem $\ref{linf}$ delivers stronger bounds than just a convexity estimate combined with Theorem $\ref{main}$.
\par

Furthermore, even in three dimensions, while our results strengthen various previous known bounds, they are still far away from the conjectured estimates. For instance, in view of Lemma $\ref{lowerbd}$ and Conjecture $\ref{dsreq}$, we see that Theorem $\ref{main}$ misses the conjectured bound by a factor of  at most $|A|^{1/2 + 2^{-1} \cdot 3^{-s+2}}$. These conjectured estimates can be justified via various heuristics, the foremost of these being the discrete restriction problem for lattice points on spheres in three dimensions; we point the reader to \S2 for more details regarding this topic.
\par

We now elaborate on the methods employed in this paper. As previously mentioned, we use a combination of incidence geometric estimates, higher energy methods as developed by Schoen and Shkredov \cite{SS2011} in arithmetic combinatorics, and various other elementary combinatorial arguments. We note that this set of ideas shares some similarity with the techniques we used in \cite{Ak2020b}, where we work with subsets $\mathscr{A}$ of the parabola instead of spheres, but there are some crucial differences. The first one is that in \cite{Ak2020b}, we were able to obtain strong estimates for the so called higher energies of $\mathscr{A}$, that is, quantities counting solutions to systems of equations of the form 
\begin{equation} \label{highdef1}
 \sum_{i=1}^{s} \vec{x}_{i} = \dots = \sum_{i=sk-s+1}^{sk} \vec{x}_{i}, 
 \end{equation}
using incidence geometric and number theoretic ideas. We then utilised these bounds along with the higher energy method to break the threshold bounds for four-fold energies. In this paper, we are unable to employ this set of techniques to get as strong bounds for higher energies as we require for the higher energy method. In fact, for any finite $A \subseteq S_{4,m}$, our methods show that $E_{2,3}(A)$, that is, the number of solutions to $\eqref{highdef1}$ when $s=2$ and $k=3$ and $\vec{x}_1, \dots, \vec{x}_{6} \in A$, satisfies
\begin{equation} \label{starsdy}
 E_{2,3}(A) \ll_{\epsilon} m^{\epsilon} |A|^3, 
 \end{equation}
which does not seem to be strong enough to deliver estimates beyond the $O_{\epsilon}(m^{\epsilon}|A|^{2+1/3})$ threshold for $E_{2,2}(A)$. Moreover, this upper bound is sharp whenever $A$ is a symmetric set, by simply noting the solutions satisfying $\vec{x}_i = -\vec{x}_{i+1}$ for $i=1,3,5$. 
\par

Thus, in this paper, we do not focus on obtaining sufficiently strong bounds for higher energies of the set $A$, but instead, our main argument concerns decomposing $A$ into two subsets, say, $X$ and $Y$, each having distinct arithmetic properties. This decomposition argument ensures that $X$ does not exhibit extremally large number of solutions to equations of the form $\vec{x}_1 + \vec{x}_2 = \vec{n}$ for arbitrary values of $\vec{n}$, which, in turn, makes additive energies of $X$ amenable to a combination of incidence geometric and higher energy methods. On the other hand, we are able to ensure that the set $Y$ lies on a small number of slices of $S_{4,m}$, and so, we treat this case with purely elementary arguments by studying additive interactions amongst this collection of slices. We refer the reader to \S7 for more details regarding this argument.
\par

Finally, we can also study this problem when $A$ is chosen to be a finite, non-empty subset of $\mathcal{S}_{3, \lambda}$. This also forms a part of a larger inquiry regarding restriction theory for spheres, an area which has seen some major contributions from the recent work of Bourgain and Demeter \cite{BD2015}. In particular, we record the following result in this direction. 

\begin{theorem} \label{sdlyf}
Let $\lambda >0$ be a real number, let $s \geq 2$ be a natural number and let $A$ be a finite, non-empty subset of $\mathcal{S}_{3, \lambda}$. Then
\[  E_{s,2}(A) \ll_{\epsilon,s} |A|^{2s -2 + 2 \cdot 3^{-s}+ \epsilon} . \]
\end{theorem}

The $s=2$ case of this result was noted in \cite{Shef}, where it was remarked that one could obtain non-trivial bounds for $E_{2,2}(A)$ using point--circle incidences, and in fact, we extend this to larger values of $s$ via the means of point--sphere incidences. We point the reader to \S8 for further details regarding the history and known results associated with this topic.
\par

We now provide an outline of our paper. We use \S2 to record more applications of our work, specially those related to discrete restriction type estimates, while also discussing the sharpness of our results in context of various examples and conjectures. In \S3, we record the suitable incidence results that we will require for our results as well as record elementary properties of lattice points on the sphere. We utilise \S4 to set the the basic framework for our techniques to work in three dimensions, and we then use \S5 to prove our two main results in this setting, that is, Theorems $\ref{main}$ and $\ref{linf}$. Similarly, we dedicate \S6 to proving our first set of threshold results for $E_{s,2}(A)$ and $E_{2,3}(A)$ when $A$ is a subset of $S_{4,m}$. We then employ \S7 to prove our main result of this paper, that is, Theorem $\ref{btdec}$. Lastly, we use \S8 to record the proof of Theorem $\ref{sdlyf}$.
\par

\textbf{Notation.} In this paper, we use the Vinogradov notation, that is, we write $X \gg_{z} Y$, or equivalently $Y \ll_{z} X$, to mean $|X| \geq C_{z} |Y|$ where $C_z$ is some positive constant depending on the parameter $z$. Moreover, if $\epsilon$ has not been fixed, then whenever we write $X \ll_{\epsilon} Y$, it will mean that this bound holds for all $\epsilon >0$. For every natural number $k \geq 2$, we use boldface to denote vectors $\vec{x} = (x_1, x_2, \dots, x_k) \in \mathbb{R}^k$, and given vectors $\vec{x}, \vec{y} \in \mathbb{R}^k$, we write $\vec{x} \cdot \vec{y} = \sum_{i=1}^{k} x_iy_i$. Additionally, for every $\theta \in \mathbb{R}$, we use $e(\theta)$ to denote $e^{2\pi i \theta}$. Next, given a non-empty, finite set $Z$, we use $|Z|$ to denote the cardinality of $Z$, and finally, whenever we are working with a fixed set $A$, we will omit the parameter $A$ from the function $r_{s}(A,\cdot)$ and simply write $r_{s}$ to be the representation function associated with $A$.

\textbf{Acknowledgements.} The author is grateful for support and hospitality from University of Bristol and Purdue University. The author would like to thank Trevor Wooley for his guidance and encouragement. The author would also like to thank Cosmin Pohoata and Joshua Zahl for helpful comments and discussions.


\section{Further discussion and applications to discrete restriction estimates}

We use this section to discuss various applications, examples and conjectures associated with our results from \S1. We will begin by producing some lower bounds for $E_{s,2}(A)$ and related quantities, when $A$ is chosen to be $S_{d,m}$ for $d\in \{3,4\}$. We will then show that a well studied problem in discrete harmonic analysis known as the discrete restriction conjecture for lattice points on the sphere implies that these lower bounds are expected to be sharp, up to $O_{\epsilon}(m^{\epsilon})$ factors. We will then demonstrate how to proceed conversely, that is, how our own results on additive energies furnish various types of discrete restriction type estimates. Finally, we record some discussion contrasting the case of spheres and paraboloids.
\par

Thus, we now establish lower bounds for $E_{s,2}(A)$ and $r_{s}(A)$ in the case when $A=S_{d,m}$. In this endeavour, it is worth noting the various estimates we have for $|S_{3,m}|$, and in particular, it is well known that $|S_{d,m}| \ll_{\epsilon} m^{d/2-1 + \epsilon}$. As for lower bounds, work of Legendre implies that the set $S_{3,m}$ is non-empty if and only if $n$ is not of the form $4^{a}(8b+7)$ for non-negative integers $a$ and $b$. Moreover, if $m \not\equiv 0,4,7 \pmod 8$, then $|S_{3,m}| \gg_{\epsilon} m^{1/2 - \epsilon}$  (see \cite{BSR2016}). For our purposes, it is also useful to know that for each odd natural number $m$, we have $|S_{4,m}| \gg m$ (see \cite[Theorem $386$]{HW1979}). This naturally leads us to the following lemma.
\par

\begin{lemma} \label{lowerbd}
Let $m$ be a natural number satisfying $m \not\equiv 0,4,7 \pmod 8$. Then, for each $s \geq 3$, we have
\[ r_{s}(S_{3,m})   \gg_{s, \epsilon} |S_{3,m}|^{s- 3- \epsilon} \ \text{and} \ E_{s,2}(S_{3,m})   \gg_{s, \epsilon} |S_{3,m}|^{2s - 3- \epsilon} . \]
Similarly, let $m$ be an odd natural number, and let $s\geq 2$. Then
\[ E_{s,2}(S_{4,m}) \gg_{s} |S_{4,m}|^{2s-2}. \]
\end{lemma}
\begin{proof}
Noting the preceding discussion, we see that assuming $m \not\equiv 0,4,7 \pmod 8$ ensures that $m^{1/2 - \epsilon} \ll_{\epsilon} |S_{3,m}| \ll_{\epsilon} m^{1/2 + \epsilon}$. Moreover, we have $s S_{3,m} \subseteq \mathbb{Z}^3 \cap [-s\sqrt{m}, s\sqrt{m}]^3,$ whence, $|s S_{3,m}| \ll_{s} m^{3/2}.$ We put these together with a double counting argument to get
\[ |s S_{3,m}| r_{s}(S_{3,m})  \geq \sum_{\vec{n} \in s S_{3,m}} r_{s,S_{3,m}}(\vec{n}) = |S_{3,m}|^{s}, \]
and so, we obtain
\[ r_{s}(S_{3,m}) \gg_{s} |S_{3,m}|^{s} m^{-3/2}  \gg_{s, \epsilon} |S_{3,m}|^{s- 3- \epsilon}.\]
 Similarly, combining our bounds for $|S_{3,m}|$ and $|sS_{3,m}|$ along with a straightforward application of the Cauchy-Schwarz inequality, we find that
\[ E_{s,2}(S_{3,m})  \geq |S_{3,m}|^{2s} |s S_{3,m}|^{-1} \gg_{s} |S_{3,m}|^{2s} m^{-3/2} \gg_{s, \epsilon} |S_{3,m}|^{2s - 3- \epsilon} . \]
\par

Likewise, in the four dimensional case, we see that the sumset $sS_{4,m}$ satisfies
\[ sS_{4,m} \subseteq \mathbb{Z}^4 \cap [-s\sqrt{m}, s\sqrt{m}]^4, \]
and consequently, we get $|s S_{4,m}| \ll_{s} m^{2}$. Combining this with another application of the Cauchy--Schwarz inequality, we obtain the bound
\[ E_{s,2}(S_{4,m}) \geq |S_{4,m}|^{2s} |s S_{4,m}|^{-1} \gg_{s} |S_{4,m}|^{2s}  m^{-2} \gg_{s} |S_{4,m}|^{2s - 2}.  \qedhere \]
\end{proof}

The estimates provided by Lemma $\ref{lowerbd}$ are expected to be sharp, up to factors of $m^{\epsilon}$, and in fact, the four dimensional estimates are known to be so, from a result of Bourgain and Demeter (see \cite[Theorem $5.1$]{BD2015a}). We describe these estimates as a part of the more general phenomenon known as the discrete restriction conjecture for lattice points on the sphere. This was originally studied by Bourgain \cite{Bo1993a}, who noted its connections with eigenfunctions of the Laplacian on $\mathbb{T}^d$ and made the following conjecture.

\begin{Conjecture} \label{dscres}
For every natural number $d \geq 3$, every complex number $\frak{a}_{\vec{n}}$  with $\vec{n} \in S_{d,m}$, every $\epsilon >0$ and every $p \geq \frac{2d}{d-2}$, we have
\[ \int_{[0,1)^d}  \big|\sum_{\vec{n} \in S_{d,m} } \frak{a}_{\vec{n}} e(\vec{\alpha} \cdot \vec{n})\big|^{p} d \vec{\alpha}  \ll_{p,\epsilon} m^{\frac{(d-2)p}{4} - \frac{d}{2} + \epsilon} \big(\sum_{\vec{n} \in S_{d,m} } |\frak{a}_{\vec{n}}|^2\big)^{p/2}. \]
\end{Conjecture}

When $d=3$, we can choose $p = 6$ and $\frak{a}_{\vec{n}}=1$ for each $\vec{n} \in A$, where $A$ is some arbitrary, non-empty subset of $S_{3,m}$, and subsequently use orthogonality to derive the conjectured upper bound $E_{3,2}(A) \ll_{\epsilon} m^{\epsilon} |A|^{3}.$ Likewise, when $d=4$,  Conjecture $\ref{dscres}$ allows us to set $p=4$, in which case, we may choose $\frak{a}_{\vec{n}}=1$ for each $\vec{n} \in A$, where $A$ is some arbitrary, non-empty subset of $S_{4,m}$. As before, applying orthogonality, we deduce that $E_{2,2}(A) \ll_{\epsilon} m^{\epsilon} |A|^{2}\,$ which is the result conjectured in Conjecture $\ref{bdconsp}$. We note that both these conjectured estimates can be extended to larger values of $s$ by a simple application of the triangle inequality and we record this as follows.

\begin{Conjecture} \label{dsreq}
Let $m$ be a natural number, let $A$ be a non-empty subset of $S_{3,m}$, and let $s \geq 3$. Then 
\begin{equation} \label{con1}
 E_{s,2}(A) \ll_{s, \epsilon} m^{\epsilon} |A|^{2s - 3}. 
 \end{equation}
Similarly, let $A$ be a non-empty subset of $S_{4,m}$ and let $s \geq 2$. Then
\begin{equation} \label{con2}
 E_{s,2}(A) \ll_{s, \epsilon} m^{\epsilon} |A|^{2s-2}. 
 \end{equation}
\end{Conjecture}

Moreover, noting Lemma $\ref{lowerbd}$, we see that the above estimates are expected to be sharp, up to factors of $O_{\epsilon}(m^{\epsilon})$. We remark that Conjecture $\ref{dscres}$ has been a subject of major work recently (see the discussion surrounding \cite[Conjecture $2.6$]{BD2015}), and the best known results for this problem arise from the recent breakthrough work of Bourgain and Demeter \cite{BD2015} on the $l^2$ decoupling conjecture. In particular, they show that whenever $d \geq 4$, Conjecture $\ref{dscres}$ holds for all $p \geq 2(d-1)/(d-3)$.
\par

As for the conjectured bound $\eqref{con2}$,  whenever $s\geq 3$, the previous best known estimates in this direction seem to arise from the aforementioned results of Bourgain and Demeter \cite{BD2015, BD2015a}, which imply that
\[ E_{s,2}(A) \ll_{\epsilon, s } m^{\epsilon} \min\{ m|A|^{2s - 3}, |A|^{2s - 2 + 1/3}\}. \]
Thus, Theorem $\ref{btdec}$ performs provides a power saving over the above bound whenever we have 
\[ |A| \ll m^{1/(1 + 6^{-1} + (1-c/2) \cdot 6^{-s+1})}.\]
 This, for example, includes the range when $|A| \ll m^{\frac{6}{7 + 6^{-s+2}}}.$
Furthermore, as mentioned before, using much more involved number theoretic methods, Bourgain and Demeter \cite{BD2015a} showed that $\eqref{con2}$ holds when $A=S_{4,m}$, which, in turn, gives us
\[ E_{2,2}(A) \leq E_{2,2}(S_{4,m}) \ll_{\epsilon} m^{2 + \epsilon}. \]
Hence, in the case when $s=2$, Theorem $\ref{btdec}$ delivers a power saving over the above estimate whenever $|A| \ll m^{1/(1+ 1/6 - c/3)}$. This, for instance, happens whenever $|A| \ll m^{6/7}$. 
Lastly, in the three dimensional case, Theorem $\ref{main}$ presents the best known estimates towards the conjectured bound $\eqref{con1}$. 
\par

Through the preceding paragraphs, we have seen that discrete restriction estimates deliver bounds for the corresponding additive energies in a straightforward manner (see also \cite[Proof of Theorem $1.4$]{Ak2020}). It has been noted in several works \cite{BD2015a, BD2015, GG2019} that this type of an implication can be reversed as well. We now use their results to convert our bounds on additive energies into discrete restriction estimates. Thus, for each finite, non-empty subset $S$ of $\mathbb{R}^d$, we write $E_{s,2}(S)$ to count the number of solutions to $\eqref{abad}$ with $\vec{x}_1, \dots, \vec{x}_{2s} \in S$, and with this definition in hand, we record the following lemma.
\par

\begin{lemma} \label{bopie}
Let $d,s$ be natural numbers such that $s \geq 2$, let $S$ be some finite, non-empty subset of $\mathbb{R}^d$, and for each $\vec{n} \in S$, let $\frak{a}_{\vec{n}}$ be a complex number. Let $\delta_{s}\in [s,2s]$ and $C = C_{s,S}>0$ be some constants such that one has $E_{s,2}(X) \leq C |X|^{\delta_{s}}$ for each $X \subseteq S$. Then for any real number $R>0$ that is large enough in terms of $S$, we have
\[ \Big( R^{-d} \int_{[0,R)^d} \big| \sum_{\vec{n} \in  S } \frak{a}_{\vec{n}} e(\vec{\alpha} \cdot \vec{n}) \big|^{2s} d \vec{\alpha} \Big)^{1/2s}  \ll_{s} C^{1/2s} (\log |S|)^{1- \delta_{s}/2s} \big( \sum_{\vec{n} \in S} |\frak{a}_{\vec{n}}|^{2s/\delta_{s}} \big)^{\delta_s/2s}. \]
\end{lemma}
\begin{proof}
Given an element $\vec{a} = (\frak{a}_{\vec{n}})_{\vec{n} \in S}$ lying in $\mathbb{C}^{|S|}$, we consider the sublinear function $T: \mathbb{C}^{|S|} \to [0, \infty)$ defined as 
\[ T(\vec{a}) = \Big( R^{-d} \int_{[0,R]^d} \big| \sum_{\vec{n} \in S} \frak{a}_{\vec{n}} e(\vec{\alpha} \cdot \vec{n}) \big|^{2s} d \vec{\alpha} \Big)^{1/2s}. \]
Moreover, for each $X \subseteq S$, we define $\vec{a}_{X} = (\frak{a}_{X,\vec{n}})_{\vec{n} \in S}$ by taking $\frak{a}_{X,\vec{n}} = 1$ when $\vec{n} \in X$ and $\frak{a}_{X,\vec{n}} = 0$ when $\vec{n} \in S\setminus X$. Note that, as in the proof of \cite[Theorem $2.8$]{BD2015}, we have
\[ T(\vec{a}_{X})^{2s} \ll  |X|^{2s}R^{-1} \nu^{-1} + E_{s,2}(X),  \]
where 
\[ \nu = \min\big\{\big|\sum_{i=1}^{s} (\vec{a}_{i}- \vec{a}_{i+s})\big| \ : \ \vec{a}_i \in S \ \text{and} \ \sum_{i=1}^{s} (\vec{a}_{i}- \vec{a}_{i+s}) \neq 0\big\}  \]
measures the ``additive geometry'' of $S$. Thus using the hypothesis that $|X|^{s} \leq E_{s}(X) \leq C |X|^{\delta_{s}}$ and taking $R \gg |S|^{2s} \nu^{-1}$, we find that
\[ T(\vec{a}_{X}) \ll_{s} C^{1/2s} |X|^{\delta_{s}/2s}. \]
We can now use \cite[Lemma $3.1$]{GG2019} to see that for all $\vec{a} \in \mathbb{C}^{|S|}$, we have
\[ T(\vec{a}) \ll_{s} C^{1/2s}(\log |S|)^{1- \delta_{s}/2s} \big(\sum_{\vec{n} \in S} |\frak{a}_{\vec{n}}|^{2s/\delta_{s}}\big)^{\delta_s/2s},\]
which delivers the desired result.
\end{proof}

We see that the conclusions of Theorems $\ref{btdec}$ and $\ref{main}$ combine naturally with Lemma $\ref{bopie}$ and periodicity of integrals to deliver the following discrete restriction estimates for $S_{3,m}$ and $S_{4,m}$.

\begin{Corollary} \label{restrict}
Let $s \geq 2$ be a natural number and let $\frak{a}_{\vec{n}}$ be a complex number for each $\vec{n} \in S_{4,m}$. Then, writing $\delta_{s} = 2s- 2 + 1/6 + (1- c)\cdot 6^{-s+1}$ and $c = 1/461$, we have
\[ \int_{[0,1)^4} \big| \sum_{\vec{n} \in  S_{4,m} } \frak{a}_{\vec{n}} e(\vec{\alpha} \cdot \vec{n}) \big|^{2s} d \vec{\alpha} \ll_{s, \epsilon} m^{\epsilon}  \big(\sum_{\vec{n} \in S_{4,m}} |\frak{a}_{\vec{n}}|^{2s/\delta_{s}}\big)^{\delta_s}. \]
Similarly, let $s \geq 2$ be a natural number and let $\frak{a}_{\vec{n}}$ be a complex number for each $\vec{n} \in S_{3,m}$. Then, writing $\gamma_{s} = 2s -3 + 1/2 + 2^{-1} \cdot 3^{-s+2}$, we have
\[ \int_{[0,1)^3} \big| \sum_{\vec{n} \in  S_{3,m} } \frak{a}_{\vec{n}} e(\vec{\alpha} \cdot \vec{n}) \big|^{2s} d \vec{\alpha} \ll_{s, \epsilon} m^{\epsilon}  \big(\sum_{\vec{n} \in S_{3,m}} |\frak{a}_{\vec{n}}|^{2s/\gamma_{s}}\big)^{\gamma_s}. \]
\end{Corollary}

Moreover, by replacing Theorem $\ref{main}$ with Theorem $\ref{sdlyf}$, we can also obtain such estimates in the case when $S_{3,m}$ is replaced by arbitrary subsets of the more general surface $\mathcal{S}_{3, \lambda}$. 

\begin{Corollary} \label{restrict2}
Let $\lambda >0$ be real, let $s \geq 2$ be a natural number, let $A$ be a finite, non-empty subset of $\mathcal{S}_{3, \lambda}$, and let $\frak{a}_{\vec{n}}$ be a complex number for each $\vec{n} \in A$. Then for any real number $R>0$ that is large enough in terms of $A$, we have
\[ R^{-3} \int_{[0,R)^3} \big| \sum_{\vec{n} \in  A } \frak{a}_{\vec{n}} e(\vec{\alpha} \cdot \vec{n}) \big|^{2s} d \vec{\alpha}  \ll_{s, \epsilon} |A|^{\epsilon}  \big(\sum_{\vec{n} \in A} |\frak{a}_{\vec{n}}|^{2s/\nu_{s}}\big)^{\nu_s}, \]
 where $\nu_{s} = 2s -2 + 2 \cdot 3^{-s}$.
\end{Corollary}

In fact, Lemma $\ref{bopie}$ implies the standard heuristic that results akin to Conjecture $\ref{dsreq}$ can be used to prove the $d=3$ and $d=4$ cases of Conjecture $\ref{dscres}$, thus providing further evidence towards the connection between additive energies and discrete restriction estimates.
\par

We finally discuss the differences between the case of spheres and paraboloids in $4$ dimensions. Recalling definition $\eqref{pardef}$ of $P_{4,m}$, we see that $|P_{4,m}| \gg m^{3}$, as well as that for every $s \geq 2$, we have
\[ sP_{4,m} \subseteq \mathbb{Z}^4 \cap ([-sm, sm]^3 \times [-3sm^2, 3sm^2]). \]
This implies that $|sP_{4,m}| \ll_{s} m^{5} \ll |P_{4,m}|^{5/3}$, and so, we may apply the Cauchy-Schwarz inequality to get
\[ E_{s,2}(P_{4,m}) \gg_{s} |P_{4,m}|^{2s - 2 + 1/3},\]
whenever $s \geq 2$. Since $\eqref{bdbnd}$ also holds for subsets $A$ of $P_{4,m}$ (see \cite[Remark $3.2$]{BD2015a}), we deduce that the above expression is sharp up to factors of $O_{\epsilon}(m^{\epsilon})$. Thus, the methods involved in proving Theorem $\ref{btdec}$ are able to differentiate between the case of $S_{4,m}$ and $P_{4,m}$.
\par

We further elaborate on this fact by recalling that a crucial ingredient in our proof is Lemma $\ref{hpitw}$, which implies that any three distinct translates of $S_{4,m}$ intersect in at most $O_{\epsilon}(m^{\epsilon})$ elements of $\mathbb{Z}^4$. Moreover, this is precisely what distinguishes the setting of paraboloids and spheres for our purposes, since Lemma $\ref{hpitw}$ does not hold true for $P_{4,m}$. In order to see this, let $P_{1} = P_{4,m} + (u, 0, 0, 3u^2)$ and $P_{2} = P_{4,m} + (0,v,0,3v^2),$ where $u$ and $v$ are natural numbers to be fixed later. We aim to show that for specific choices of $u$ and $v$, we have $|P_{4,m} \cap P_1 \cap P_2| \gg m$. We note that
\[ P_{1} \cap P_{4,m} = \{ \vec{x} \in P_{4,m} \ | \ x_1 = 2u \} \ \text{and} \ P_{2} \cap P_{4,m} = \{ \vec{x} \in P_{4,m} \ | \ x_2 = 2v\}, \]
and so, choosing $u = v =1$, we get
\[ P_{4,m} \cap P_1 \cap P_2 = \{ (2, 2, n, 8 + n^2) \ | \ -m \leq n \leq m \}. \]
This yields the desired conclusion $|P_{4,m} \cap P_1 \cap P_2| \gg m \gg |P_{4,m}|^{1/3}$.


\section{Preliminaries}

We begin by recording some convenient notation. For every real number $p >0$, finite set $Z$ and function $w : \mathbb{Z} \to \mathbb{R}$, we write
\[ \n{w}_{p} = (\sum_{z \in Z} |w(z)|^{p})^{1/p} \ \text{and} \ \n{w}_{\infty} = \sup_{z \in Z} |w(z)|. \]
Similarly, for any vector $\vec{z} \in \mathbb{Z}^d$, we define $\n{\vec{z}}_{p} = (|z_1|^p + \dots + |z_d|^p)^{1/p}$ and $\n{\vec{z}}_{\infty} = \sup_{1 \leq i \leq d} |z_i|$.
Finally, we will use $\mathscr{C}_{d, k}$ to denote the set of hyperplanes $H$ in $\mathbb{R}^d$ such that $H$ is of the form $a_1 x_1 + \dots + a_d x_d = a_{d+1}$ with $a_1, \dots, a_{d+1}$ lying in the set $[-k,k] \cap \mathbb{Z}$. With this notation in hand, we commence by analysing the distribution of lattice points over intersections of hyperplanes from $\mathscr{C}_{d, k}$ and the sphere $\mathcal{S}_{d,m}$. 

\begin{lemma} \label{jarnik}
Let $c>0$ be a real number. Then for every $H \in \mathscr{C}_{3,m^{c}}$, we have
\[  |H \cap S_{3,m}| \ll_{c,\epsilon} m^{\epsilon}. \]
Similarly, let $H_1$ and $H_2$ be two distinct hyperplanes in $\mathscr{C}_{4,m^{c}}$. Then, we have
\[ |H_1 \cap H_2 \cap S_{4,m}| \ll_{c, \epsilon} m^{\epsilon}. \]
\end{lemma}

\begin{proof}
Without loss of generality, we may assume that whenever $(x,y,z) \in H \cap S_{3,m}$, then the point $(x,y)$ lies on the ellipsoid
\begin{equation} \label{yui21}
 a_1 x^2 + a_2 xy + a_3 y^2 + a_4 x + a_5 y + a_6 = 0, 
 \end{equation}
for an appropriate choice  of $a_1, \dots, a_6 \in \mathbb{Z}$ satisfying $|a_1|, \dots, |a_6| \ll m^{O_{c}(1)}$. We may then follow \cite[Proof of Lemma $3.2$]{She2016} mutatis mutandis to see that this curve contains at most $e^{O_{c}(\log m/\log \log m)} \ll_{c,\epsilon}m^{\epsilon}$ elements of $\mathbb{Z}^2$, whence, we are done. In the four dimensional case, we may assume that $H_1$ and $H_2$ are not translates of each other, since otherwise we would have $H_1 \cap H_2 = \emptyset$. This implies that $H_1$ and $H_2$ intersect in a two dimensional plane, which further intersects $S_{4,m}$ in a circle. Moreover, we see that every lattice point on this circle corresponds to a distinct lattice point on an ellipsoid of the form $\eqref{yui21}$, for some suitable choice of $a_1, \dots, a_6 \in \mathbb{Z}$ with $|a_1|, \dots, |a_6| \ll m^{O_{c}(1)}$. We may now proceed as in the three dimensional setting to obtain the desired bound.
\end{proof}

We also present some preliminary definitions from incidence geometry. Thus, given a finite set $P$ of points and a finite collection $V$ of varieties of bounded degree in $\mathbb{R}^d$, we write that the incidence graph of $P\times V$ is $K_{s,t}$-free to mean that any $s$ distinct points in $P$ can simultaneously lie on at most $t-1$ distinct varieties in $V$. In particular, we will be interested in studying the number of incidences between $P$ and $V$, and we will use $I(P,V)$ to denote this quantity. Thus, 
\[ I(P,V) = \sum_{p \in P} \sum_{v \in V} \mathds{1}_{p \in v} = |\{ (p,v) \in P \times V \ | \ p \in v \} |. \]
Moreover, we will also study weighted incidences between $P$ and $V$, and so, given functions $w:P \to \mathbb{N}$ and $w': V \to \mathbb{N}$, we write
\[ I_{w,w'}(P,V) = \sum_{p \in P} \sum_{v \in V} w(p) w'(v) \mathds{1}_{p \in v}  .\]
We now record the required incidence theorem from \cite{FPS2017} as described in \cite[Theorem B.1]{BM2019}. 

\begin{lemma} \label{st1}
Let $P$ be a finite set of points in $\mathbb{R}^d$ and let $V$ be a finite collection of varieties in $\mathbb{R}^d$ of degree at most $k$, and let $\epsilon >0$ be a real number. Assuming the incidence graph of $P\times V$ is $K_{s,t}$-free, we have
\[ I(P,V) \ll_{k,s,d,\epsilon} s t (|P|^{\frac{(d-1)s}{ds-1} + \epsilon} |V|^{\frac{d(s-1)}{ds-1}}  + |P| + |V|). \]
\end{lemma}

It is worth noting that in all of the cases where we intend to apply Lemma $\ref{st1}$, we will have $t \ll_{\epsilon} m^{\epsilon}$ and $|P| \leq m^{O(1)}$, and so, $t|P|^{\epsilon} \ll_{\epsilon} m^{\epsilon}$. Thus, we will often coalesce the $t|P|^{\epsilon}$ factor in the above incidence bound into the $O_{\epsilon}(m^{\epsilon})$ notation. In fact, because of these properties, it would also have sufficed to use \cite[Theorem 12.4]{Shef}, which is another quantitative version of Lemma $\ref{st1}$. Furthermore, since we will often be working with weighted incidence sums, we record the following lemma that applies a dyadic decomposition type argument to convert incidence bounds, such as the one above, into weighted incidence estimates.

\begin{lemma} \label{6on}
Let $d$ be a natural number, let $P$ be a finite, non-empty set of points in $\mathbb{R}^d$ and let $L$ be a finite, non-empty set of varieties. Suppose there exist real numbers $\mathcal{C}>0$ and $a, b \in (1/2,1)$ such that
\begin{equation} \label{6onh}
I(P',L') \leq \mathcal{C} ( |P'|^{a} |L'|^{b} + |P'| + |L'|),
\end{equation}
holds for each non-empty $P' \subseteq P$ and $L' \subseteq L'$. Then, for every function $w: P \to \mathbb{N}$ and $w': L \to \mathbb{N}$, we have
\[ I_{w,w'}(P,L) \ll  \mathcal{C}( \n{w}_{2}^{2-2a} \n{w}_1^{2a-1} \n{w'}_{2}^{2-2b} \n{w'}_{1}^{2b-1} + \n{w}_{\infty} \n{w'}_{1} + \n{w}_1 \n{w'}_{\infty}). \]
\end{lemma}

\begin{proof}
Our argument proceeds via a dyadic decomposition argument. Thus, let $J$ be the largest natural number such that $2^{J} \leq \n{P}_{\infty}$, and similarly, let $K$ be the largest natural number such that $2^K \leq  \n{L}_{\infty}$. Furthermore, for each $1 \leq j \leq J$  and for each $1 \leq k \leq K$, we write
\[ P_{j} = \{ p \in P \ | \ 2^{j} \leq w(p) < 2^{j+1} \}, \ \text{and} \ L_{k} = \{ l \in L \ | \ 2^{k} \leq w'(l) < 2^{k+1} \}. \]
\par

We see that
\[  \sum_{p \in P} \sum_{l \in L} \mathds{1}_{p \in l} w(p)w'(l)  
= \sum_{j=0}^{J} \sum_{k=0}^{K} \sum_{p \in P_{j}}\sum_{l \in L_{k}}\mathds{1}_{p \in l} w(p)w'(l)  
\ll \sum_{j=0}^{J} \sum_{k=0}^{K} 2^{j}2^{k} \sum_{p \in P_{j}}\sum_{l \in L_{k}}\mathds{1}_{p \in l}  . \]
Applying $\eqref{6onh}$ with $P' = P_{j}$ and $L'= L_{k}$, we find that
\begin{equation} \label{onwith}
  \sum_{p \in P} \sum_{l \in L} \mathds{1}_{p \in l} w(p)w'(l) 
\ll C  \sum_{j=0}^{J} \sum_{k=0}^{K} 2^{j}2^{k} (|P_{j}|^{a}|L_{k}|^{b} + |P_j| + |L_k|)  .
  \end{equation}
\par

We note that
 \[ \sum_{j=0}^{J} \sum_{k=0}^{K} 2^{j}2^{k} |P_{j}| = \big(\sum_{j=0}^{J} 2^{j}|P_{j}| \big) \big(\sum_{k=0}^{K} 2^{k}\big) \ll \n{w}_1 2^{K} \ll \n{w}_1 \n{w'}_{\infty}, \]
and
 \[\sum_{j=0}^{J} \sum_{k=0}^{K} 2^{j}2^{k} |L_{k}| = \big(\sum_{j=0}^{J} 2^{j}\big) \big(\sum_{k=0}^{K} 2^{k} |L_{k}| \big)  \ll 2^{J} \n{w'}_{1} \ll \n{w}_{\infty} \n{w'}_{1}. \]
Furthermore, we observe that
\[ \sum_{j=0}^{J} \sum_{k=0}^{K} 2^{j}2^{k} |P_j|^{a}|L_{k}|^{b} = \big(  \sum_{j=0}^{J}2^j|P_j|^{a}  \big) \big( \sum_{k=0}^{K} 2^k |L_k|^{b}     \big),\]
and so, noting $\eqref{onwith}$, it suffices to show that
\begin{equation} \label{app24}
\sum_{j=0}^{J}2^j|P_j|^{a}   \ll \n{w}_1^{2a-1} \n{w}_2^{2-2a}  \ \text{and} \ \sum_{k=0}^{K} 2^k |L_k|^{b}   \ll  \n{w'}_1^{2b-1} \n{w'}_2^{2-2b}.
\end{equation}
\par

Setting $X =  \n{w}_2^2 \n{w}_1^{-1}$, we write $U = \{ 0 \leq j \leq J \ | \ 2^{j} \leq X\}$ and $V = \{ 0,1,\dots, J\} \setminus U$. We note that
\[ \sum_{j \in U} 2^j|P_j|^{a}  = \sum_{j \in U}( |P_{j}|2^{j})^{a} 2^{j(1-a)} \leq \n{w}_1^{a} \sum_{j \in U} 2^{j(1-a)} , \]
and since $a <1$, we have $\sum_{j \in U} 2^{j(1-a)} \ll X^{1-a},$ whence, 
\begin{equation} \label{6hon}
  \sum_{j \in U} 2^j|P_j|^{a} \ll \n{w}_1^{a} X^{1-a}  = \n{w}_1^{2a-1} \n{w}_2^{2-2a} . 
\end{equation}
\par
  
Next, we consider the case when $j \in V$, and so, we have
\[ \sum_{j \in V} 2^{j} |P_{j}|^{a}  = \sum_{j \in V} (|P_j|^{a} 2^{2aj}) 2^{-j(2a-1)} \leq \n{w}_2^{2a} \sum_{j \in V} 2^{-j(2a-1)}. \]
Since $a > 1/2$, we observe that $\sum_{j \in V} 2^{-j(2a-1)} \ll X^{-(2a-1)},$ whereupon, the preceding inequality gives us
\[ \sum_{j \in V} 2^{j} |P_{j}|^{a} 
 \ll \n{w}_2^{2a} X^{-(2a-1)}=  \n{w}_1^{2a-1} \n{w}_2^{2-2a}. \]
Combining this with $\eqref{6hon}$, we find that
\[ \sum_{j=0}^{J}2^j|P_j|^{a}   \ll   \n{w}_1^{2a-1} \n{w}_2^{2-2a} , \]
which confirms the first inequality in $\eqref{app24}$. We can prove the second inequality in $\eqref{app24}$ mutatis mutandis, and so, we conclude the proof of Lemma $\ref{6on}$.
\end{proof}

Since our aim will be to study additive energies for sets on the sphere, it is natural to analyse, given a vector $\vec{n} \in \mathbb{Z}^d$, the set $C_{\vec{n}} = S_{d, m} \cap (\vec{n} - S_{d, m})$. We note that when $\vec{n} \neq 0$, the set $C_{\vec{n}}$ is contained in a unique hyperplane $H_{\vec{n}}$ which does not contain the origin. In order to see this, observe that $H_{\vec{n}}$ is orthogonal to the vector $\vec{n}$ and contains the vector $\vec{n}/2$, and so, $H_{\vec{n}}$ is defined by the equation $\vec{x} \cdot \vec{n} = 2^{-1} \n{\vec{n}}_2^2$.

\begin{lemma} \label{hpitw}
Let $d \in \{3,4\}$ be a natural number, let $c>0$ be a real number and let $\vec{a}_1, \dots, \vec{a}_{d-1} \in \mathbb{Z}^d$ be distinct vectors satisfying $\n{\vec{a}_i}_{\infty} \leq m^{c}$ for $1 \leq i \leq d-1$. Then
\[ |(\vec{a}_1 + S_{d,m}) \cap \dots \cap (\vec{a}_{d-1} + S_{d,m})| \ll_{c,\epsilon} m^{\epsilon}. \]
\end{lemma}

\begin{proof}
Since our intersection of interest is translation invariant, we may assume that $\vec{a}_{d-1} = 0$ and that $\vec{a}_1, \dots, \vec{a}_{d-2}$ are distinct non-zero vectors. When $d=3$, we need to show that whenever $\vec{a} \neq 0$, the set $S_{3,m} \cap (S_{3,m} + \vec{a})$ has $O_{\epsilon}(m^{\epsilon})$ elements.  In particular, note that $S_{3,m} \cap (S_{3,m} + \vec{a}) \subseteq S_{3,m} \cap H_{\vec{a}}$, whence, we use Lemma $\ref{jarnik}$ to see that $|S_{3,m} \cap H_{\vec{a}}| \ll_{c,\epsilon} m^{\epsilon}$ and thus, we obtain the desired conclusion. 
\par

When $d=4$, we study $|S_{4,m} \cap (S_{4,m} + \vec{a}) \cap (S_{4,m} + \vec{b})|$ for distinct non-zero vectors $\vec{a}, \vec{b}$. As in the preceding paragraph, we note that
\[ S_{4,m} \cap (S_{4,m} + \vec{a}) \cap (S_{4,m} + \vec{b}) \subseteq S_{4,m} \cap H_{\vec{a}} \cap H_{\vec{b}}. \]
If $\vec{b} = \lambda \vec{a}$ for some $\lambda \in \mathbb{R} \setminus \{0\}$, then any $\vec{u} \in H_{\vec{a}} \cap H_{\vec{b}}$ satisfies
\[ \vec{a} \cdot \vec{u}  = 2^{-1} \n{ \vec{a}}_2^2 \ \text{and} \  \lambda \vec{a} \cdot \vec{u} =2^{-1}  \lambda^2 \n{ \vec{a}}_2^2. \]
This implies that $\lambda  = \lambda^2$, whence, $\lambda \in \{0,1\}$, but either choice contradicts the fact that $\vec{a}, \vec{b}$ were distinct non-zero vectors. We may now assume that $\vec{b} \neq \lambda \vec{a}$ for any $\lambda \in \mathbb{R}\setminus \{0\}$, which confirms that $H_{\vec{a}}$ and $H_{\vec{b}}$ are distinct hyperplanes from $\mathscr{C}_{d, 2m^{2c}}$. We obtain the required bound by using Lemma $\ref{jarnik}$ to deduce that $|S_{4,m} \cap H_{\vec{a}} \cap H_{\vec{b}}| \ll_{c,\epsilon} m^{\epsilon}$. 
\end{proof}

This lemma allows us to obtain some control over how many translates of a sphere can contain a fixed set of lattice points.

\begin{Corollary} \label{mnf2}
Let $d \in \{3,4\}$, let $c>0$, let $P,X$ be arbitrary, finite sets of points in $\mathbb{Z}^d$ such that every $\vec{p} \in P$ satisfies $\n{\vec{p}}_{\infty} \leq m^{c}$, and let $V$ be
\[ V = \{ \mathcal{S}_{d, m} + \vec{x} \ | \ \vec{x} \in X\}. \]
Then the incidence graph of $P \times V$ is $K_{d-1, t}$-free for some $t \ll_{c,\epsilon} m^{\epsilon}$. 
\end{Corollary}

\begin{proof}
Let $\vec{a}_1, \dots, \vec{a}_{d-1}$ be fixed points in $P$ and let $\vec{u} \in X$ satisfy $\vec{a}_1, \dots, \vec{a}_{d-1} \in \mathcal{S}_{d,m} + \vec{u}$. Since $\vec{a}_1 ,\dots, \vec{a}_{d-1}, \vec{u} \in \mathbb{Z}^d$, we have $\vec{a}_1, \dots, \vec{a}_{d-1} \in S_{d,m} + \vec{u}$. This implies that $\vec{u} \in \vec{a}_{i} + S_{d,m}$ for each $i = 1, \dots, d-1$, whereupon, we may apply Lemma $\ref{hpitw}$ to deduce that there are $O_{c,\epsilon}(m^{\epsilon})$ valid choices for $\vec{u}$. 
\end{proof}

As we previously mentioned, our strategy will involve applying higher energy variants of the Balog--Szemer\'edi--Gowers type theorem. In particular, we will be using the following result of Shkredov \cite[Theorem $1.3$]{Sh2013}. 

\begin{lemma} \label{shkbs}
Let $A$ be a finite,non-empty subset of an abelian group, let $K,M$ satisfy \[ E_{2,2}(A) = |A|^3/K \ \text{and} \ E_{2,3}(A) = M|A|^{4}/K^2. \]
Then there exists a set $A' \subseteq A$ such that
\begin{equation} \label{ffcl}
 |A'| \gg M^{-10} (\log M)^{-15} |A| \ \text{and} \ |2A'-A'| \ll M^{162} (\log M)^{252}  K |A'|. 
 \end{equation}
\end{lemma}


\section{Set up for the three dimensional case}

We begin this section by presenting some preliminary estimates for the additive energy of finite subsets of $S_{3,m}$ using a combination of Lemma $\ref{jarnik}$ and elementary combinatorial arguments. Thus, we fix a non-empty subset $A$ of $S_{3,m}$, and define the additive energy $E_{s,2}(A)$ and the representation function $r_{s}=r_{s}(A, \cdot)$ as in \S1.

\begin{lemma} \label{trives}
For any $\vec{n} \neq 0$, we have $r_{2}(\vec{n}) \ll_{\epsilon} m^{\epsilon}$, and $r_{2}(0) \ll |A|$, and $r_{3}(0) \ll_{\epsilon}m^{\epsilon}|A|$. Moreover, when $s \geq 3$, we have
\[ r_{s}(A) \ll_{\epsilon} m^{\epsilon}|A|^{s-2}. \]
Lastly, when $s \geq 2$, we have
\[ E_{s,2}(A) \ll_{\epsilon} m^{\epsilon}|A|^{2s-2}. \]
\end{lemma}

\begin{proof}
We note that $r_{2}(0) \leq |A|$ trivially, and that $r_{s}(\vec{n}) \neq 0$ if and only if $\vec{n} \in sA \subseteq sS_{3,m}$, whence, $\n{\vec{n}}_{\infty} \ll_{s} m$. Thus, when $\vec{n} \neq 0$, we see that 
\begin{equation} \label{luke2}
r_{2}(\vec{n}) \leq |C_{\vec{n}}| = |H_{\vec{n}} \cap S_{3,m}| \ll_{\epsilon} m^{\epsilon}, 
\end{equation}
where the last inequality follows from Lemma $\ref{jarnik}$. Furthermore, inequality $\eqref{luke2}$ delivers the bounds $r_{3}(0) \ll_{\epsilon} m^{\epsilon} |A|$, and $r_{2}(\vec{n}) \ll_{\epsilon} m^{\epsilon}$ whenever $\vec{n} \neq 0$, in a straightforward manner. When $s \geq 3$, we have
\begin{align*}
 r_{s}(0) 
 & = \sum_{\vec{a}_1, \dots, \vec{a}_{s} \in A} \mathds{1}_{\vec{a}_1 + \dots + \vec{a}_s = 0} =   \sum_{\vec{a}_1+ \dots + \vec{a}_{s-2}=0} \mathds{1}_{\vec{a}_1 + \dots + \vec{a}_s = 0} + \sum_{\vec{a}_1+ \dots + \vec{a}_{s-2}\neq 0} \mathds{1}_{\vec{a}_1 + \dots + \vec{a}_s = 0} \\
& \leq  r_{s-2}(0)r_{2}(0) + |A|^{s-2} \sup_{\vec{n} \neq 0} r_{2}(-\vec{n}) \ll_{\epsilon} m^{\epsilon}  |A|^{s-2},
\end{align*}
giving us the required conclusion. Similarly, letting $s \geq 3$ and $\vec{n} \neq 0$, we get
\[ r_{s}(\vec{n}) \leq |A|^{s-2} \sup_{\vec{n'} \neq 0} r_{2}(\vec{n'}) + r_{s-2}(0)  r_{2}(\vec{n}) + r_{s-2}(\vec{n}) r_{2}(0) , \]
which gives us $ r_{s}(\vec{n}) \ll_{\epsilon} m^{\epsilon}|A|^{s-2}.$ Next, we use $\eqref{cs1}$ to see that
\[  E_{s,2}(A) = \sum_{\vec{n}} r_{s}(\vec{n})^2 \ll_{\epsilon} m^{\epsilon} |A|^{s-2} \sum_{\vec{n}} r_{s}(\vec{n}) =  m^{\epsilon} |A|^{2s-2},\]
 whenever $s \geq 3$. Finally, when $s=2$, we have
 \[ E_{2,2}(A) = \sum_{\vec{n} \neq 0} r_{2}(\vec{n})^2 + r_{2}(0)^2 \ll_{\epsilon} m^{\epsilon} |A|^2. \qedhere \] 
 \end{proof}
 
We will now record the specific incidence estimate that we will be using to prove Theorems $\ref{main}$ and $\ref{linf}$.

\begin{lemma} \label{wtst}
Let $P$ be a set of points in $\mathbb{Z}^3$ with an associated weight function $w : P \to \mathbb{N}$, and let $L$ be a set of lattice point translates of the sphere $\mathcal{S}_{3,m}$ with an associated weight function $w' : L \to \mathbb{N}$. Suppose that $C>0$ satisfies $\n{\vec{p}}_{\infty} \leq m^{C}$ for every $\vec{p} \in P$. Then 
\[ I_{w,w'}(P,L)   \ll_{C,\epsilon} m^{\epsilon} (\n{w}_2^{2/5} \n{w}_{1}^{3/5} \n{w'}_{2}^{4/5} \n{w'}_1^{1/5} + \n{w}_1 \n{w'}_{\infty} + \n{w}_{\infty} \n{w'}_{1}). \]
\end{lemma}

\begin{proof}
We see that $|P|^{\epsilon} \ll m^{3C\epsilon}$, and so, we put together Corollary $\ref{mnf2}$ along with Lemma $\ref{st1}$ applied with the parameters $s=2$ and $d=3$ and $t \ll_{\epsilon} m^{\epsilon}$ to discern that for every $P' \subseteq P$ and for every $L' \subseteq L$, we have
\[ I(P',L') \ll_{\epsilon,C} m^{\epsilon}( |P'|^{4/5} |L'|^{3/5} + |P'| + |L'|), \]
which, in turn, combines with Lemma $\ref{6on}$ to furnish the desired conclusion.
\end{proof}


\section{Proofs of Theorems $\ref{main}$ and $\ref{linf}$}

We use this section to prove Theorems $\ref{main}$ and $\ref{linf}$. We begin by recording a lemma that estimates $E_{s,2}(A)$ in terms of $s,m, |A|$ and $E_{s-1,2}(A)$.

\begin{lemma} \label{iter}
Let $s\geq 3$. Then, we have
\[ E_{s,2}(A) \ll_{s,\epsilon} m^{\epsilon} ( |A|^{(4s-3)/3} E_{s-1,2}(A)^{1/3}  + |A|^{2s-3} ).\]
\end{lemma}

\begin{proof}
We write
\[ E_{s,2}(A) = \sum_{\vec{a}_1, \dots, \vec{a}_{2s} \in A} \mathds{1}_{\vec{a}_2 + \dots + \vec{a}_s = \vec{a}_{s+1} + \dots + \vec{a}_{2s} - \vec{a}_1}. \]
We first count the contribution from solutions of the form
\begin{equation} \label{in1}
 0 = \vec{a}_2 + \dots + \vec{a}_s = \vec{a}_{s+1} + \dots + \vec{a}_{s} - \vec{a}_1 , 
\end{equation}
with $\vec{a}_i \in A$ for $1 \leq i \leq 2s$. In particular, this is bounded above by $r_{s-1}(0)r_{s+1}(0).$ Using Lemma $\ref{trives}$, we see that when $s =3$, this is bounded above by $O_{\epsilon}(m^{\epsilon}|A|^3)$, and when $s\geq 4$, this is bounded above by $O_{\epsilon}(m^{\epsilon} |A|^{2s-3})$. In either case, we can bound the contribution from solutions satisfying $\eqref{in1}$ by $O_{\epsilon}(m^{\epsilon}|A|^{2s-3})$. 
\par

Similarly, we consider contribution from solutions satisfying
\begin{equation} \label{in2}
0 = \vec{a}_{s+1} + \dots + \vec{a}_{2s} = \vec{a}_1 + \vec{a}_2 + \dots + \vec{a}_{s},
\end{equation}
with $\vec{a}_i \in A$ for $1 \leq i \leq 2s$. In particular, this is bounded above by $r_{s}(0)^2$. Using Lemma $\ref{trives}$, we deduce that when $s \geq 3$, we have $r_{s}(0)^2 \ll_{\epsilon} m^{\epsilon}|A|^{2s-4}$, whence, we can bound the contribution from solutions satisfying $\eqref{in2}$ by $O_{\epsilon}(m^{\epsilon}|A|^{2s-4})$. 
\par

Thus, it suffices to prove that
\begin{equation} \label{pyr}
 \sum_{\substack{\vec{a}_2 + \dots + \vec{a}_s \neq 0,  \\ \vec{a}_{s+1} + \dots +  \vec{a}_{2s} \neq 0}} \mathds{1}_{\vec{a}_2 + \dots + \vec{a}_s = \vec{a}_{s+1} + \dots + \vec{a}_{2s} - \vec{a}_1}  \ll_{\epsilon} m^{\epsilon}( |A|^{(4s-3)/3} E_{s-1,2}(A)^{1/3}  + |A|^{2s-3} ). 
 \end{equation}
Setting $\vec{a} = \vec{a}_1, \ \text{and} \  \vec{u} = \vec{a}_2 + \dots + \vec{a}_{s} \ \text{and} \ \vec{v} = \vec{a}_{s+1} + \dots + \vec{a}_{2s},$ the sum on the left hand side of $\eqref{pyr}$ can be rewritten as 
\[ \sum_{\vec{v} \in sA \setminus \{0\}} \sum_{\vec{u} \in (s-1)A \setminus \{0\}} \sum_{\vec{a} \in A}  r_{s-1}(\vec{v}) r_{s}(\vec{u}) \mathds{1}_{\vec{u} = \vec{v}-\vec{a}}. \]
Writing $l_{\vec{v}}  =  \vec{v} - \mathcal{S}_{3,m}$ for each $\vec{v} \in \mathbb{R}^3$, we see that the above sum can be bounded above by $I_{w,w'}(P,L)$, where $P = (s-1) A \setminus \{0\}$ and $w(\vec{n}) = r_{s-1}(\vec{n})$ for each $\vec{n} \in P$, and 
\[ L = \{ l_{\vec{v}} \ | \ \vec{v} \in s A \setminus \{0\} \}\]
and $w'(l_{\vec{v}}) = r_{s}(\vec{v})$ for each $l_{\vec{v}} \in L$. 
Moreover, since $P \subset (s-1)A \subseteq (s-1)S_{3,m}$, we have $\n{\vec{p}}_{\infty} \ll s m^{1/2}$ for every $\vec{p} \in P$, and so, we can use Lemma $\ref{wtst}$ to deduce that
\[ I_{w,w'}(P,L) \ll_{s, \epsilon} m^{\epsilon} (\n{w}_2^{2/5} \n{w}_{1}^{3/5} \n{w'}_{2}^{4/5} \n{w'}_1^{1/5} + \n{w}_1 \n{w'}_{\infty} + \n{w}_{\infty} \n{w'}_{1}). \]
\par

We note that
\[ \n{w}_1 = |A|^{s-1} \ \text{and} \ \n{w}_2^2 = E_{s-1,2}(A) \ \text{and} \ \n{w}_{\infty} \ll_{\epsilon} m^{\epsilon} |A|^{s-3}, \]
with the last inequality following from Lemma $\ref{trives}$. Similarly, we have
\[ \n{w'}_1 = |A|^{s} \ \text{and} \ \n{w'}_2^2 = E_{s,2}(A) \ \text{and} \ \n{w'}_{\infty} \ll_{\epsilon} m^{\epsilon} |A|^{s-2}. \]
Combining these estimates with our incidence bound, we obtain 
\[ E_{s,2}(A) \ll_{s,\epsilon} m^{\epsilon}(E_{s-1,2}(A)^{1/5} |A|^{(3s-3)/5} E_{s,2}(A)^{2/5}|A|^{s/5} + |A|^{2s-3}).\]
Upon simplifying the above bound, we find that
\[ E_{s,2}(A) \ll_{s,\epsilon} m^{\epsilon} ( |A|^{(4s-3)/3} E_{s-1,2}(A)^{1/3}  + |A|^{2s-3} ), \]
whence, we conclude the proof of Lemma $\ref{iter}$. 
\end{proof}

We will now present Theorem $\ref{main}$ as a straightforward corollary of Lemma $\ref{iter}$. 

\begin{proof}[Proof of Theorem $\ref{main}$]
Our proof proceeds via induction. First, setting $s=2$, we use Lemma $\ref{trives}$ to see that $E_{2,2}(A) \ll_{\epsilon} m^{\epsilon} |A|^{2},$ in which case, we are done. We now move to the inductive step, and so, letting $s \geq 3$, we assume that
\[ E_{s-1,2}(A) \ll_{s,\epsilon} |A|^{2s -5+ 1/2 + 2^{-1} \cdot 3^{-s+3}}.\]
Applying Lemma $\ref{iter}$, we get
\begin{align*}
 E_{s,2}(A) & \ll_{s,\epsilon} m^{\epsilon}(|A|^{2s/3 -5/3+ 1/6 + 2^{-1} \cdot 3^{-s+2}}|A|^{(4s-3)/3} + |A|^{2s-3}) \\
& \ll_{s,\epsilon} m^{\epsilon}|A|^{2s - 3 + 1/2 + 2^{-1} \cdot 3^{-s+2}}.
\end{align*}
This finishes the inductive step, and so, we see that Theorem $\ref{main}$ holds for all $s \geq 2$.
\end{proof}

Our next goal is to prove Theorem $\ref{linf}$, and this will be our main aim throughout the rest of this section. 

\begin{proof}[Proof of Theorem $\ref{linf}$]
Let $\vec{n} \in sA$. We first consider the case when $s \geq 5$ is an odd natural number. Thus, writing $s = 2p+1$ for some $p \geq 2$, we have
\[ r_{s}(\vec{n}) = \sum_{\vec{a}_1, \dots, \vec{a}_{2p+1} \in A} \mathds{1}_{\vec{a}_1 + \dots + \vec{a}_{p} = \vec{n} - \vec{a}_{p+1} - \dots - \vec{a}_{2p}  - \vec{a}_{2p+1}}. \]
As before, we count the contribution of terms of the form 
\[ 0 =  \vec{a}_1 + \dots + \vec{a}_{p} = \vec{n} - \vec{a}_{p+1} - \dots - \vec{a}_{2p}  - \vec{a}_{2p+1}, \]
with $\vec{a}_i \in A$ for $1 \leq i \leq 2p+1$. This contribution can be estimated to be $r_{p}(0) r_{p+1}(\vec{n})$. Using Lemma $\ref{trives}$, when $p=2$, this is bounded above by $O_{\epsilon}(m^{\epsilon}|A|^2)$, and when $p \geq 3$, this is bounded above by $O_{\epsilon}(m^{\epsilon}|A|^{2p -3})$. Thus, for all $p \geq 1$, we can estimate the number of such solutions by $O_{\epsilon}(m^{\epsilon}|A|^{2p - 2})$. Similarly, we consider the contribution of solutions satisfying
\[ \vec{a}_1 + \dots + \vec{a}_{p} - \vec{n} + \vec{a}_{2p+1} =  - \vec{a}_{p+1} - \dots - \vec{a}_{2p}  =  0, \]
with $\vec{a}_i \in A$ for $1 \leq i \leq 2p+1$. We can bound this by $r_{p}(0)r_{p+1}(\vec{n})$, which, in turn, we can estimate by $O_{\epsilon}(m^{\epsilon}|A|^{2p - 2})$. 
\par

Since $|A|^{2p-2} = |A|^{s-3}$, it suffices to consider the sum
\[ \sum_{\substack{\vec{a}_1 + \dots + \vec{a}_p \neq 0, \\ \vec{a}_{p+1} + \dots + \vec{a}_{2p} \neq 0}} \mathds{1}_{\vec{a}_1 + \dots + \vec{a}_{p} = \vec{n} - \vec{a}_{p+1} - \dots - \vec{a}_{2p}  - \vec{a}_{2p+1}}. \]
As before, we can now rewrite the above sum as 
\[ \sum_{\vec{u} \in pA \setminus \{0\} } \sum_{\vec{v} \in pA\setminus \{0\}} \sum_{\vec{a} \in A} r_{p}(\vec{u})r_{p}(\vec{v}) \mathds{1}_{\vec{u} = \vec{n} - \vec{v} - \vec{a}}, \]
whereupon, noting the fact that $\vec{a} \in A \subseteq \mathcal{S}_{3,m}$, this can be further bounded above by 
\[ \sum_{\vec{u} \in pA \setminus \{0\} } \sum_{\vec{v} \in pA\setminus \{0\}} \sum_{\vec{a} \in A} r_{p}(\vec{u})r_{p}(\vec{v}) \mathds{1}_{\vec{u} \in  l_{\vec{v}}'},   \]
where for each $\vec{v} \in pA\setminus \{0\}$, we define $l_{\vec{v}}' = \vec{n} - \vec{v} - \mathcal{S}_{3,m}$. In particular, this is a weighted incidence count between the set of points $p A \setminus \{0\}$ and the set $L = \{ l_{\vec{v}}'  \ | \ \vec{v} \in p A \setminus \{0\}\}$, where for every $\vec{u} \in p A \setminus \{0\}$, we have $w(\vec{v}) = r_{p}(\vec{u})$ and $\n{\vec{u}}_{\infty} \ll p m^{1/2}$ while for every $l_{\vec{v}}' \in L$, we have $w'(l_{\vec{v}}' ) = r_{p}(\vec{v})$. Thus, combining this with Lemma $\ref{wtst}$, we see that
\[  r_{s}(\vec{n}) \ll_{s,\epsilon} m^{\epsilon} ( E_{p}(A)^{3/5}|A|^{4p/5} + |A|^{2p-2}) \ll_{s,\epsilon} m^{\epsilon} |A|^{2p-2 + 1/2 + 3\eta_{p}/5} . \]
This gives us
\[ r_{s}(\vec{n}) \ll_{s,\epsilon} m^{\epsilon} |A|^{s - 3 + 1/2 + \lambda_{s}} , \]
where $\lambda_{s} =  10^{-1} \cdot 3^{-(s-1)/2+3}$.
\par

In the case when $s = 2p$, we can proceed similarly as above to prove our result. Instead, we will deduce bounds of the same quality in a more direct fashion using estimates for the additive energy $E_{p}(A)$. In particular, using orthogonality, we can write
\[ r_{s}(\vec{n}) = \int_{[0,1)^3} \big(\sum_{\vec{u} \in A} e(\vec{\alpha} \cdot \vec{u})\big)^{2p} e(-\vec{n}\cdot\vec{\alpha}) d \vec{\alpha}. \]
Applying triangle inequality on the right hand side above and then using orthogonality again, we find that
\[ r_{s}(\vec{n}) \leq  \int_{[0,1)^3} \big|\sum_{\vec{u} \in A} e(\vec{\alpha} \cdot \vec{u})\big|^{2p} d \vec{\alpha} = E_{p,2}(A). \]
Combining the above bound with Theorem $\ref{main}$, we see that
\[ r_{s}(\vec{n}) \leq E_{p,2}(A) \ll_{p, \epsilon} m^{\epsilon} |A|^{2p-3 + 1/2 + \eta_{p}}. \]
Since $2p = s$ and $\eta_{p} = 2^{-1} 3^{-s/2 + 2} = \lambda_{s},$ we get the desired conclusion.
\end{proof}


\section{Incidence geometric estimates for additive energies when $d=4$}

In this section, we obtain our first set of bounds for $E_{2,3}(A)$ and $E_{s,2}(A)$ when $s \geq 2$ and $A$ is some non-empty subset of $S_{4,m}$. Moreover, this would involve studying elements $\vec{n} \in \mathbb{R}^4$ that have many representations as a sum of two elements from our set $A$. Thus, for each $\tau \geq 1$, we let 
\[ P_{\tau} = \{ \vec{n} \in \mathbb{R}^d \ | \ \tau \leq r_{2}(\vec{n})  < 2\tau\}. \]
We see that for every non-zero $\vec{n} \in P_{\tau}$, the set $C_{\vec{n}}$, and consequently, the hyperplane $H_{\vec{n}}$ contains at least $\tau$ elements of $A$. Hence, we let $\mathcal{H}_{\tau} = \{ H_{\vec{n}} \ | \ \vec{n} \in P_{\tau} \setminus \{0\}\}$.
\par

We now use a variant of the hyperplane trick used by Bourgain and Demeter to derive incidence geometric upper bounds for $|P_{\tau}|$ (see, for instance, \cite{BD2015a}, \cite{Shef}).  
\begin{Proposition} \label{kz2}
For each $\tau \geq 1$, we have 
\[ |P_{\tau}| \tau \ll_{\epsilon} m^{\epsilon} ( |P_{\tau}|^{\frac{6}{7}} |A|^{\frac{4}{7}} + |P_{\tau}| + |A|). \]
\end{Proposition}
\begin{proof}
We begin by noting that when $|P_{\tau}| \leq 1$, Proposition $\ref{kz2}$ holds trivially, and so, we can assume that $|P_{\tau}| \geq 2$. This, in turn, implies that $|P_{\tau}| \ll |P_{\tau} \setminus \{0\}| \ll |P_{\tau}|$, and as a result, it suffices to prove the desired inequality for the set $P_{\tau} \setminus \{0\}$. Furthermore, we claim that there are at most $O_{\epsilon}(m^{\epsilon})$ elements from $A$ lying simultaneously in $2$ distinct hyperplanes from $\mathcal{H}_{\tau}$. In order to see this, let the $2$ distinct hyperplanes be $H_{\vec{n}_1}, H_{\vec{n}_2}$. Since $A \cap H_{\vec{n}} \subseteq S_{4,m} \cap (\vec{n}+S_{4,m})$, we have
\[ A \cap H_{\vec{n}_1} \cap H_{\vec{n}_{2}} \subseteq S_{4,m} \cap (\vec{n}_1+S_{4,m}) \cap (\vec{n}_{2} + S_{4,m}), \]
for some non-zero, distinct $\vec{n}_1, \vec{n}_2 \in A+A$, whence, we can use Lemma $\ref{hpitw}$ to deduce the suitable claim.
\par

With this in hand, we see that the incidence graph of $A \times \mathcal{H}_{\tau}$ is $K_{t,2}$-free for some $t \ll_{\epsilon} m^{\epsilon}$. We now move to the dual space, that is, for any point $\vec{a} \in A$, we define the hyperplane 
\[ G_{\vec{a}} =\{ \vec{x} \in \mathbb{R}^4 \ | \  \vec{x} \cdot \vec{a} = 1 \}. \]
Similarly, since no hyperplane $H$ in $\mathcal{H}_{\tau}$ contains the origin, we may write every $H_{\vec{n}} \in \mathcal{H}_{\tau}$ as $\vec{x} \cdot \vec{u}_{\vec{n}} =1$ for some $\vec{u}_{\vec{n}}$. In this case, we define the point $p_{\vec{n}} = \vec{u}_{\vec{n}}$. Note that $\mathds{1}_{\vec{a} \in H_{\vec{n}}} = \mathds{1}_{p_{\vec{n}} \in G_{\vec{a}}},$ and so, writing 
\[ P = \{p_{\vec{n}} \ | \ \vec{n} \in P_{\tau} \setminus \{0\} \} \ \text{and} \ \mathcal{H}' = \{ G_{\vec{a}} \ | \ \vec{a} \in A\}, \]
we see that
\[ \sum_{\vec{a} \in A} \sum_{  H_{\vec{n}} \in \mathcal{H}_{\tau}   } \mathds{1}_{   \vec{a} \in H_{\vec{n}}   } =  \sum_{  p_{\vec{n}} \in P} \sum_{   G_{\vec{a}} \in \mathcal{H}'   } \mathds{1}_{  p_{\vec{n}} \in G_{\vec{a}}  }.\]
Moreover, since the incidence graph of $A \times \mathcal{H}_{\tau}$ is $K_{t,2}$-free for some $t \ll_{\epsilon} m^{\epsilon}$, we infer that the incidence graph of $P \times \mathcal{H}'$ is $K_{2, t}$-free for some $t \ll_{\epsilon} m^{\epsilon}$.
\par

Thus, we may apply Lemma $\ref{st1}$ with the parameters $s=2$ and $d= 4$ and $t \ll_{\epsilon} m^{\epsilon}$, along with the fact that $|P| = |P_{\tau} \setminus \{0\}| \leq |A+A| \leq m^{4}$, to obtain the inequality
\[  \sum_{\vec{a} \in A} \sum_{  H_{\vec{n}} \in \mathcal{H}_{\tau}   } \mathds{1}_{   \vec{a} \in H_{\vec{n}}   } =   \sum_{  p_{\vec{n}} \in P} \sum_{   G_{\vec{a}} \in \mathcal{H}'   } \mathds{1}_{  p_{\vec{n}} \in G_{\vec{a}}  } \ll_{\epsilon} m^{\epsilon} (|P|^{  \frac{6}{7}} |\mathcal{H}'|^{   \frac{4}{7}} + |P| + |\mathcal{H}'| ). \]
Finally, since each $H_{\vec{n}} \in \mathcal{H}_{\tau}$ contains at least $\tau$ elements of $A$, we have
\[ \sum_{\vec{a} \in A} \sum_{  H_{\vec{n}} \in \mathcal{H}_{\tau}   } \mathds{1}_{   \vec{a} \in H_{\vec{n}}   }  \geq |\mathcal{H}_{\tau} | \tau = |P_{\tau}\setminus \{0\}| \tau. \]
Combining this with the preceding inequality gives us the desired conclusion.
\end{proof}

We will now use these incidence methods to obtain our threshold bounds for $E_{2,2}(A)$ and $E_{2,3}(A)$, where
\begin{equation} \label{doingaton}
 E_{2,i}(A) = \sum_{\vec{n} \in 2A} r_{2}(\vec{n})^i \ll \sum_{1 \leq 2^j \leq |A|} |P_{2^j}| 2^{ij},  
 \end{equation}
for $i \in \mathbb{N}$. Our main idea is to roughly divide into two cases, that is, when $2^j \leq |A|^{\frac{1}{3}}$ and when $2^j > |A|^{\frac{1}{3}}$. In the former case, we can use the trivial inequality
\[ |P_{2^j}|2^{ji} \leq |A|^{\frac{i}{3}} \sum_{\vec{n}} r_{2}(\vec{n}) \leq |A|^{2+\frac{i}{3}}, \]
for $i \in \{2,3\}$. In the latter case, we will utilise Proposition $\ref{kz2}$ to obtain the required bound.

\begin{lemma} \label{floma}
Let $A \subseteq S_{4,m}$ be a finite set. Then 
\[ E_{2,2}(A) \ll_{ \epsilon} m^{\epsilon} |A|^{2+\frac{1}{3}}, \ \text{and} \ E_{2,3}(A) \ll_{\epsilon} m^{\epsilon} (|A|^{2+\frac{2}{3}}+|A| \sup_{\vec{n}}r_{2}(\vec{n})^2). \]
\end{lemma}
\begin{proof}
We begin by observing that whenever $\tau \ll_{\epsilon} m^{\epsilon}$, then 
\[ |P_{\tau}| \tau^{i} \ll_{\epsilon} m^{2 \epsilon} |P_{\tau}| \tau \leq m^{2\epsilon} |A|^2, \]
for $i =2,3$. Hence, noting $\eqref{doingaton}$, whenever $2^j \ll_{\epsilon} m^{\epsilon}$, the contribution of $|P_{2^j}|2^{ji}$ to $E_{2,i}(A)$ is bounded above by $O_{\epsilon}(m^{2\epsilon} |A|^2 \log |A|)$, which is stronger than the required estimate. Thus, we may assume that $2^j \geq C_{\epsilon} m^{\epsilon}$ for some sufficiently large constant $C_{\epsilon}$. In this case, Proposition $\ref{kz2}$ implies that
\begin{equation} \label{ptau1}
 |P_{2^j}|  \ll_{\epsilon} m^{\epsilon} (|A|^{ 4}2^{-{7j}} + |A| 2^{-j}). 
 \end{equation}
Thus, we have
\[ |P_{2^j}| 2^{2j} \ll_{\epsilon} m^{\epsilon} (|A|^{ 4}2^{-{5j}}+ |A| 2^j) \ll m^{\epsilon} (|A|^{ 4}2^{-{5j}} + |A|^2). \]
\par

We now let $\Delta= |A|^{1/3}$, and so, we see that
\[ \sum_{2^{j} \leq \Delta} |P_{2^j}|2^{2j} \leq |A|^2 \Delta = |A|^{2 + \frac{1}{3}}. \]
Furthermore, the preceding discussion implies that
\begin{align*}
\sum_{2^j > \Delta} |P_{2^j}|2^{2j} & \ll_{\epsilon} m^{\epsilon} \sum_{2^j > \Delta} (|A|^{ 4}2^{-{5j}} + |A|^2) \ll_{\epsilon} m^{\epsilon} (|A|^4 \Delta^{ -5 } + |A|^{2} \log |A|) \ll_{\epsilon} m^{\epsilon} |A|^{2 + 1/3}, 
\end{align*}
where the last inequality follows from substituting the value of $\Delta$. Putting this together with the preceding set of inequalities, we get
\[ \sum_{j} |P_{2^j}| 2^{2j} 
\leq \sum_{2^j \leq C_{\epsilon} m^{\epsilon}}  |P_{2^j}| 2^{2j} + \sum_{C_{\epsilon} m^{\epsilon} < 2^j \leq \Delta} |P_{2^j}|2^{2j} + \sum_{C_{\epsilon} m^{\epsilon}, \Delta < 2^j }|P_{2^j}|2^{2j}
 \ll_{\epsilon} m^{\epsilon} |A|^{2+1/3}. \]
 We get the required bound for $E_{2,2}(A)$ by combining the above estimate with $\eqref{doingaton}$.
\par

We begin our analysis of third energies now, and so, we note that
\[ \sum_{2^{j} \leq \Delta} |P_{2^j}|2^{3j} \leq |A|^2 \Delta^2 \leq |A|^{2+\frac{2}{3}}. \]
Furthermore, we may use $\eqref{ptau1}$ to infer that
\begin{align*}
\sum_{2^j > \Delta} |P_{2^j}| 2^{3j} 
 \ll_{\epsilon} m^{\epsilon} \sum_{2^j > \Delta} (|A|^{4}2^{-4j} + |A|2^{2j})  \ll_{\epsilon} m^{\epsilon} ( |A|^{4} \Delta^{-4} + |A| \sup_{\vec{n}} r_{2}(\vec{n})^2). 
\end{align*}
Substituting $\Delta = |A|^{1/3}$ in the above expression, we get
\[\sum_{2^j > \Delta} |P_{2^j}| 2^{3j}  \ll_{\epsilon} m^{\epsilon} (|A|^{2 + 2/3} + |A| \sup_{\vec{n}} r_{2}(\vec{n})^2). \]
As in the case of $E_{2,2}(A)$, we see that this is sufficient to prove that
\[ E_{2,3}(A)  \ll_{\epsilon} m^{\epsilon}  (|A|^{2+2/3} +|A| \sup_{\vec{n}}r_{2}(\vec{n})^2), \]
and so, we are done with our proof of Lemma $\ref{floma}$.
\end{proof}

We are now interested in obtaining upper bounds for $E_{s,2}(A)$, when $s \geq 3$ and $A \subseteq S_{4,m}$. We begin this endeavour by establishing the incidence bound that we will require. Thus, given finite sets $P, X \subseteq \mathbb{Z}^4$, putting Corollary $\ref{mnf2}$ together with Lemma $\ref{st1}$ applied with the parameters $s=3$ and $d=4$ and $t \ll_{\epsilon} m^{\epsilon}$, we get that
\[ I(P,V) \ll_{\epsilon, \mathcal{C}} m^{\epsilon}( |P|^{\frac{9}{11}} |V|^{\frac{8}{11}} + |P| + |V|), \]
whenever every $\vec{p} \in P$ satisfies $\n{\vec{p}}_{\infty} \leq m^{\mathcal{C}}$ and $V = \{ \vec{x} + \mathcal{S}_{4,m} \ | \ \vec{x} \in X\}$. Furthermore, we can combine this with Lemma $\ref{6on}$ to see that
\begin{align} \label{wtdinc}
I_{w,w'}(P,V)   \ll_{\epsilon,\mathcal{C}} m^{\epsilon} ( \n{w}_2^{4/11} \n{w}_1^{7/11} \n{w'}_2^{6/11} \n{w'}_1^{5/11} + \n{w}_1 \n{w'}_{\infty} + \n{w}_{\infty} \n{w'}_1). 
 \end{align}
\par

\begin{lemma}  \label{zee11}
Let $s \geq 2$ and let $A$ be a finite subset of $S_{4,m}$. Then, we have
\[ E_{s,2}(A) \ll_{\epsilon,s} m^{\epsilon}     ( |A|^{\frac{12s - 7}{8}}   E_{s-1,2}(A)^{\frac{1}{4}} + |A|^{2s-2}), \]
\end{lemma}
\begin{proof}
We begin by writing
\begin{align*} 
E_{s,2}(A) = \sum_{\vec{a}_1, \dots, \vec{a}_{2s} \in A} \mathds{1}_{\vec{a}_1 + \dots + \vec{a}_{s} - \vec{a}_{s+1} = \vec{a}_{s+2} + \dots + \vec{a}_{2s}} = \sum_{\vec{a} \in A} \sum_{\vec{v} \in s A} \sum_{\vec{u} \in (s-1)A} r_{s}(\vec{v}) r_{s-1}(\vec{u})\mathds{1}_{ \vec{v} - \vec{a} = \vec{u}}.
\end{align*}
The latter can be bounded above by an incidence bound, and so, we write $P= (s-1)A$ and $L  = \{ l_{\vec{v}} \ | \ \vec{v} \in sA\}$, where $l_{\vec{v}} = \mathcal{S}_{4, m} + \vec{v}$ for each $\vec{v} \in sA$. Thus, we have
\[ \sum_{\vec{a} \in A} \sum_{\vec{v} \in s A} \sum_{\vec{u} \in (s-1)A} r_{s}(\vec{v}) r_{s-1}(\vec{u})\mathds{1}_{\vec{v}- \vec{a} = \vec{u}} \leq 
\sum_{\vec{p} \in P} \sum_{l_{\vec{v}} \in L} r_{s}(\vec{v}) r_{s-1}(\vec{u}) \mathds{1}_{\vec{p} \in l_{\vec{v}}} . \]
We can bound the right hand side above using $\eqref{wtdinc}$ and the fact that $\n{\vec{p}}_{\infty} \ll_{s} m^{s}$ for every $\vec{p} \in (s-1)A$, whence,
\[ E_{s,2}(A) \ll_{\epsilon,s} m^{\epsilon} ( E_{s-1,2}(A)^{\frac{2}{11}}|A|^{\frac{7s-7}{11}} E_{s,2}(A)^{\frac{3}{11}} |A|^{\frac{5s}{11} } + |A|^{2s-2}). \]
Simplifying the above inequality furnishes the bound
\[ E_{s,2}(A) \ll_{\epsilon,s} m^{\epsilon}     ( |A|^{\frac{12s - 7}{8}}   E_{s-1,2}(A)^{\frac{1}{4}} + |A|^{2s-2}), \]
which is the desired conclusion.
\end{proof}

Therefore, when $s=3$, Lemmata $\ref{floma}$ and $\ref{zee11}$ combine to deliver the bound
\[ E_{3,2}(A) \ll_{\epsilon, s} m^{\epsilon} (|A|^{29/8} E_{2,2}(A)^{1/4} + |A|^{4}) \ll_{\epsilon, s} m^{\epsilon} |A|^{4 +5/24}. \]
We may now use Cauchy-Schwarz inequality to see that
\begin{equation} \label{sio2}
 |2A-A| \geq |A|^{6} E_{3,2}(A)^{-1} \gg_{\epsilon} m^{-\epsilon} |A|^{2-5/24}.
 \end{equation}
\par

We end this section by mentioning that $\eqref{sio2}$ is a consequence of a combination of incidence geometric ideas utilised along with Lemma $\ref{hpitw}$, and it plays an important role in our proof of Theorem $\ref{btdec}$ (see proof of Lemma $\ref{haptg}$). Moreover, as per our discussion in \S2, inequality $\eqref{sio2}$ does not hold when $A$ is replaced by the set $P_{4,m}$.


\section{Proof of Theorem $\ref{btdec}$}

We dedicate this section to proving Theorem $\ref{btdec}$, using our bounds for many-fold sumsets along with the higher energy method. We begin by noting that it suffices to show that Theorem $\ref{btdec}$ holds in the specific case when $s=2$, which we record as follows.

\begin{theorem} \label{bt22}
Let $A$ be a non-empty subset of $S_{4,m}$ and let $\delta = 1/2766$. Then
\[ E_{2,2}(A) \ll_{\epsilon} m^{\epsilon} |A|^{2 + 1/3 - \delta}. \]
\end{theorem}

\begin{proof}[Proof of Theorem $\ref{btdec}$]
We prove our result inductively, and our base case will be $s=2$, which is handled by Theorem $\ref{bt22}$. Thus, we move to the inductive step, and so we suppose that $s \geq 3$ and that
\[ E_{s-1,2}(A) \ll_{\epsilon} m^{\epsilon} |A|^{2s -4 + 1/6 + (1 - c)\cdot 6^{-s+2}}. \]
We substitute this in the conclusion of Lemma $\ref{zee11}$ to get
\begin{align*}
E_{s,2}(A)  \ll_{\epsilon} m^{\epsilon}(|A|^{3s/2 - 7/8} E_{s-1,2}(A)^{1/4} + |A|^{2s-2}) \ll_{\epsilon} m^{\epsilon} |A|^{2s -2 + 1/6 +  (1 - c)\cdot 6^{-s+1}},
\end{align*}
which concludes the inductive step, as well as our proof of Theorem $\ref{btdec}$.
\end{proof}

Henceforth, we will now focus our attention towards proving Theorem $\ref{bt22}$. We commence by introducing some useful notation, and thus, for each finite subset $X$ of $A$ and for each $\vec{n} \in \mathbb{R}^4$, we write
\begin{equation} \label{ciise}
C_{\vec{n}, X} = X \cap (\vec{n} - X). 
\end{equation}
It is worth noting that $C_{\vec{n}, X} \subseteq S_{4,m} \cap (\vec{n} - S_{4,m}) = S_{4,m} \cap (\vec{n} + S_{4,m})$, and that $|C_{\vec{n}, X}| = r_{2}(X, \vec{n})$.
\par

We also go through some necessary reductions. Thus, let $E_{1} = \{0\}$ and let $E_{2} = (0, \infty)$ and let $E_{3} = (-\infty,0)$, and for each $\vec{i} \in \{1,2,3\}^4$, write
\[ E_{\vec{i}} = E_{i_1} \times E_{i_2} \times E_{i_3} \times E_{i_4}. \]
Note that we have $\mathbb{R}^4 = \cup_{\vec{i} \in \{1,2,3\}^4} E_{\vec{i}}.$ We further write $A_{\vec{i}} = A \cap E_{\vec{i}}$ and $S_{\vec{i}} = S_{4,m} \cap E_{\vec{i}}$ for each $\vec{i} \in \{1,2,3\}^4$. Our main idea is to reduce to the case when $A \subseteq E_{(2,2,2,2)}$ and in this endeavour, we record the following standard lemma from additive combinatorics (see, for instance, \cite[Exercise 2.3.20]{TV2006}). 

\begin{lemma} \label{cshol1}
Suppose $A$ is a finite, non-empty subset of $\mathbb{R}^d$, and suppose $A = A_1 \cup A_2$ where $A_1$ and $A_2$ are disjoint sets. Then $E_{2,2}(A) \ll \sup_{1 \leq i \leq 2} E_{2,2}(A_i).$
\end{lemma}

Note that when $\vec{i} = (1,1,1,1)$, then $A_{\vec{i}} = \emptyset$. Next, if $\vec{i}$ has precisely three coordinates being equal to $1$, then $E_{\vec{i}}$ is a line, and since a line intersects a sphere in $O(1)$ points, we have $|A_{\vec{i}}| \ll 1$. Similarly, when $\vec{i}$ has precisely two coordinates being equal to $1$, the set $A_{\vec{i}}$ lies on a curve of the form 
\[ x_1^2 + x_2^2 = m. \]
Here, we can use a standard estimate to deduce that $|A_{\vec{i}}| \ll_{\epsilon} m^{\epsilon}$ (see \cite[Theorem $338$]{HW1979}), and so, in both these two cases, we can use the trivial bound
\[ E_{2,2}(A_{\vec{i}}) \leq |A_{\vec{i}}|^3 \ll_{\epsilon} m^{\epsilon}. \]
When $\vec{i}$ has precisely one coordinate equalling $1$, the set $A_{\vec{i}}$ lies on a sphere of the form
\[ x_1^2 + x_2^2 + x_3^2 = m, \]
in which case, we can use Lemma $\ref{trives}$ to show that
\[ E_{2,2}(A_{\vec{i}}) \ll_{\epsilon} m^{\epsilon} |A_{\vec{i}}|^2. \]
Hence, combining this discussion with Lemma $\ref{cshol1}$, we see that it is enough to prove Theorem $\ref{bt22}$ for sets of the form $A_{\vec{i}}$, where $\vec{i}$ has no coordinate equalling $1$. Moreover, since the sets $S_{\vec{i}}$, with $\vec{i} \in \{2,3\}^4$, are equivalent up to rotation, we may assume without loss of generality that $A \subseteq E_{(2,2,2,2)}$. 

\begin{lemma} \label{ntdv}
Let $A \subseteq S_{(2,2,2,2)}$, let $N = |A|$. Then there exist disjoint subsets $X,Y$ of $A$ satisfying the following properties. Firstly, $A = X \cup Y$. Then, whenever $X \neq \emptyset$, we have 
\begin{equation} \label{prop1x}
r_{2}(X, \vec{n}) < N^{2/3 + \delta}. 
\end{equation}
for each $\vec{n} \in X+X$. Moreover, whenever $Y\neq \emptyset$, there exists a natural number $1 \leq r \leq N$ and vectors $\vec{n}_{0}, \dots, \vec{n}_{r-1}$ and sets $A_{r-1} \subseteq \dots \subseteq A_{1} \subseteq A$, such that $r \leq N^{1/3 - \delta}$ and 
\begin{equation} \label{Prop2y}
Y = \cup_{i=0}^{r-1} C_{\vec{n}_{i}, A_{i}},
\end{equation}
where the sets $C_{\vec{n}_i, A_{i}}$ are defined as in $\eqref{ciise}$ and are pairwise disjoint.
\end{lemma}

\begin{proof}
We describe an algorithm that finishes in finitely many steps and provides us with sets $X,Y$ having the desired properties. Thus, we begin by writing $A_0 = A$ and $B_0 = \emptyset$, and see that $A= A_0 \cup B_0$. Now, given a natural number $i \geq 1$, suppose that we begin the $i^{th}$ iteration, and so, let $A_{i-1}$ and $B_{i-1}$ be pairwise disjoint sets satisfying $A= A_{i-1} \cup B_{i-1}$. If there exists $\vec{n}$ such that $r_{2}(A_{i-1},\vec{n}) \geq N^{2/3 + \delta}$, let $\vec{n}_{i-1} = \vec{n}$ and let
\[ A_{i} = A_{i-1} \setminus C_{\vec{n}_{i-1}, A_{i-1}} \ \text{and} \ B_i = B_{i-1} \cup C_{\vec{n}_{i-1}, A_{i-1}}. \]
If no such $\vec{n}$ exists, we end the algorithm.
\par

Since we remove at least $N^{2/3 + \delta}$ elements from $A_{i}$ in each iteration, we can have at most $r \leq N^{1/3-\delta}$ iterations. Moreover, when $A_r$ is non-empty, we have that $r_{2}(A_r, \vec{n}) < N^{2/3 + \delta}$ for each $\vec{n} \in \mathbb{R}^4$. Similarly, it is easy to see that when the set $B_r$ is non-empty, $B_r$  satisfies the properties we desire of $Y$, and so, upon setting $X= A_r$ and $Y=B_r$, we are done.
\end{proof}

Thus, it suffices to show that
\begin{equation} \label{dvmn2}
 E_{2,2}(X) \ll_{\epsilon} m^{\epsilon} N^{2+1/3 - \delta} \ \text{and} \ E_{2,2}(Y) \ll_{\epsilon} m^{\epsilon} N^{2 + 1/3 - \delta}, 
 \end{equation}
since we can then use Lemma $\ref{cshol1}$ to obtain the required bound for $E_{2,2}(A)$. We begin by focusing on the second inequality in $\eqref{dvmn2}$.

\begin{lemma} \label{ntdv}
Let $Y \subseteq A$ be a set as in the conclusion of Lemma $\ref{ntdv}$. Then, we have
\[ E_{2,2}(Y) \ll_{\epsilon} m^{\epsilon} r |Y|^2. \]
\end{lemma}

\begin{proof}[Proof of Lemma $\ref{ntdv}$]
Let $\vec{y}_1, \dots, \vec{y}_4$ be elements of $Y$ satisfying
\begin{equation} \label{rou4}
\vec{y}_1 + \vec{y}_2 = \vec{y}_3 + \vec{y}_4.
\end{equation}
We may suppose that $\vec{y}_i \neq \vec{y}_j$ for any $1 \leq i < j \leq 4$, since solutions not satisfying this condition contribute an amount $O(|Y|^2)$ to $E_{2,2}(Y)$. Thus, we have
\[ \vec{y}_4 = \vec{y}_1 + (\vec{y}_2  - \vec{y}_3). \]
We have $O(|Y|^2)$ choices for $\vec{y}_2, \vec{y}_3$ such that $\vec{y}_2 \neq \vec{y}_3$. Moreover, we have $r$ choices for $i$ satisfying $\vec{y}_1 \in C_{\vec{n}_{i}, A_{i}}$. We can fix these parameters in $O(r|Y|^2)$ ways, and so, we must show that there are $O_{\epsilon}(m^{\epsilon})$ choices for $\vec{y}_4$. In particular, this would finish our proof, since each choice of $\vec{y}_{4}$ fixes $\vec{y}_1$ as $\vec{y}_1, \dots, \vec{y}_4$ satisfy $\eqref{rou4}$. 
\par

We note that 
\begin{equation} \label{wowmoment}
 \vec{y}_{4} \in Y \cap ( C_{\vec{n}_{i}, A_{i}} + (\vec{y}_2  - \vec{y}_3)) \subseteq S_{4,m} \cap (S_{4,m} +  \vec{y}_2  - \vec{y}_3) \cap (S_{4,m} + \vec{n}_i+ \vec{y}_2  - \vec{y}_3) ,
 \end{equation}
where the last inclusion follows from the fact that $C_{\vec{n}_{i}, A_{i}} \subseteq S_{4,m} \cap (S_{4,m} + \vec{n}_i)$. We can ignore the cases when $\vec{y}_2 - \vec{y}_3 = 0$ or $\vec{y}_2 - \vec{y}_3 = -\vec{n}_i$, since the former contradicts our assumption that $\vec{y}_2 \neq \vec{y}_3$, while the latter would imply that 
\[ \vec{y}_4 \in Y \cap (C_{\vec{n}_i, Y}  - \vec{n}_i) \subseteq Y \cap (Y - \vec{n}_i) \cap (-Y) \subseteq E_{(2,2,2,2)} \cap (-E_{(2,2,2,2)} ),\]
contradicting the fact that $E_{(2,2,2,2)}  \cap (-E_{(2,2,2,2)}) = \emptyset$. But if $\vec{y}_2 - \vec{y}_{3} \notin \{0, -\vec{n}_i\}$, then $\eqref{wowmoment}$ implies that $\vec{y}_{4}$ lies in three distinct translates of $S_{4,m}$, whenceforth, Lemma $\ref{hpitw}$ implies that there are at most $O_{\epsilon}(m^{\epsilon})$ choices for $\vec{y}_4$. Thus, we have proven that $E_{2,2}(Y) \ll_{\epsilon} m^{\epsilon} r |Y|^2$.
\end{proof}

Note that Lemma $\ref{ntdv}$ combines with the fact that $r \leq N^{1/3 - \delta}$ and $|Y| \leq |A|=N$ to deliver the second inequality in $\eqref{dvmn2}$, whereupon, it is sufficient to prove the first inequality in $\eqref{dvmn2}$. As we previously mentioned, the properties that the set $X$ satisfies makes it amenable to the higher energy method, and so, we present the following upper bound for $E_{2,2}(X)$. 

\begin{lemma} \label{haptg}
Let $X \subseteq A$ be a set as in the conclusion of Lemma $\ref{ntdv}$, and let $\epsilon >0$. Then 
\[  E_{2,2}(X) \ll_{\epsilon} m^{C\epsilon} N^{2+ 1/3 - \delta},\]
where $C>0$ is some absolute constant. 
\end{lemma}

\begin{proof}
We assume that $E_{2,2}(X) \geq  N^{2 + 1/3 - \delta}$, since otherwise we are done. Moreover, Lemma $\ref{floma}$ implies that
\[ E_{2,3}(X) \ll_{\epsilon} m^{\epsilon}(|X|^{2 + 2/3}+ |X|\sup_{\vec{n}} r_{2}(X,\vec{n})^2 ) , \]
and thus, noting $\eqref{prop1x}$, we get
\begin{equation} \label{wfhpn}
E_{2,3}(X) \ll_{\epsilon} m^{\epsilon} (|X|^{2+2/3} + |X| N^{4/3 + 2 \delta}). 
\end{equation}
If the second term on the right hand side dominates, that is, when $|X| \leq N^{4/5 + 6 \delta /5}$, then
\[ E_{2,3}(X) \ll_{\epsilon} m^{\epsilon} |X| N^{4/3 + 2 \delta}, \]
which, in turn, combines with an application of the Cauchy-Schwarz inequality to give us
\[ E_{2,2}(X) \leq |X| E_{2,3}(X)^{1/2} \ll_{\epsilon} m^{\epsilon} |X|^{3/2} N^{2/3 + \delta}  \ll_{\epsilon} m^{\epsilon} N^{2 + 1/6 + \delta} .\]
Moreover, since $\delta < 1/12$, we are done in this case.
\par

Henceforth we may assume that the first term on the right hand in $\eqref{wfhpn}$ dominates, in which case, we get
\begin{equation} \label{trh2}
 E_{2,3}(X) \ll_{\epsilon} m^{\epsilon} |X|^{2+2/3} . 
 \end{equation}
We now apply Lemma $\ref{shkbs}$ for the set $X$. Thus, we see that
\[ |X|^3/K = E_{2,2}(X) \geq N^{2+ 1/3 - \delta}, \]
whence, 
\begin{equation} \label{upbdk4}
 K \leq |X|^{3} N^{-2 -1/3} N^{\delta}.
 \end{equation}
Similarly, we have
\[ M|X|^4/K^2= E_{2,3}(X)  \ll_{\epsilon} m^{\epsilon} |X|^{2 + 2/3}, \]
which gives us
\begin{equation} \label{upbdm4}
 M \ll_{\epsilon} m^{\epsilon} |X|^{-1-1/3} K^2 \ll_{\epsilon} m^{\epsilon} |X|^{4+2/3} N^{-4-2/3} N^{2 \delta} \leq m^{\epsilon} N^{2\delta}.
 \end{equation}
\par

The conclusion of Lemma $\ref{shkbs}$ implies that there must exist $X' \subseteq X$ such that
\[ |X'| \gg M^{-10} (\log M)^{-15}|X| \ \text{and} \ |2X'-X'| \ll M^{162} (\log M)^{252}   K |X'|. \]
As before, we may use $\eqref{sio2}$ to discern that
\[|2X'-X'| \gg_{\epsilon} m^{-\epsilon} |X'|^{1 + \frac{19}{24}},\]
which, in conjunction with the preceding inequality, gives us
\[ M^{-10} (\log M)^{-15} |X|^{19/24} |X'|  \ll |X'|^{1+19/24} \ll_{\epsilon} m^{\epsilon} M^{162} (\log M)^{252}    K|X'|  .\]
Simplifying the above, we see that
\[ |X|^{19/24} \ll_{\epsilon} m^{\epsilon} M^{172} (\log M)^{267} K .  \]
We now insert our upper bounds $\eqref{upbdk4}$ and $\eqref{upbdm4}$ for $K$ and $M$ respectively in the above inequality to get
\[ |X|^{19/24} \ll_{\epsilon} m^{173\epsilon} N^{344 \delta} (\log M)^{267} |X|^{3} N^{-2-1/3}N^{\delta}.\]
This, combined with the fact that 
\[ (\log M)^{267} \ll_{\epsilon} (\log m \log N)^{267} \ll_{\epsilon} m^{\epsilon},\]
 delivers the bound
\[ N^{1/8-345 \delta}\ll_{\epsilon} m^{174 \epsilon}   . \]
Now, since $345\delta =345/2766 < 1/8 $, we discern that $N \ll_{\epsilon} m^{C\epsilon}$, for some constant $C>0$, but this yields the bound
\[ E_{2,2}(X) \leq N^3 \ll_{\epsilon} m^{3C\epsilon}, \]
and so, we get the desired conclusion anyway.
\end{proof}

Note that Lemma $\ref{haptg}$ implies the first inequality in $\eqref{dvmn2}$ after we rescale $\epsilon$ appropriately, and so, we are done with the proof of Theorem $\ref{bt22}$.


\section{Additive energies on $\mathcal{S}_{3,\lambda}$ and point--sphere incidences}

We end this paper by studying $E_{s,2}(A)$ when $A$ is some arbitrary subset of $\mathcal{S}_{3,\lambda}$, for some fixed $\lambda >0$. Since the additive equation $\eqref{abad}$ is invariant under affine transformations, we may assume that $\lambda = 1$ after a suitable dilation. In this case, we have two regimes of results, as in the setting of the parabola. The first collection of results provide upper bounds for $E_{2,2}(A)$ in terms of $|A|$ and $\delta_A$, where 
\[ \delta_{A} = \inf_{\substack{   \vec{a}_1, \vec{a}_2 \in A  \\ \vec{a}_1 \neq \vec{a}_2}  } | \vec{a}_1 - \vec{a}_2|. \]
Here, the decoupling results of Bourgain and Demeter \cite{BD2015} imply that
\begin{equation} \label{hfli}
 E_{2,2}(A) \ll_{\epsilon} \delta_{A}^{-\epsilon} |A|^{2}, 
 \end{equation}
which can then be extended to obtain the estimate
\[ E_{s,2}(A) \ll_{\epsilon} \delta_{A}^{-\epsilon} |A|^{2s-2}, \]
whenever $s \geq 2$. Moreover, $\eqref{hfli}$ is sharp so long as $\delta_A^{-1} \ll |A|^{O(1)}$, but when $\delta_A^{-1}$ is large in terms of $|A|$, say, $\delta_A^{-1} \gg 2^{2^{|A|}}$, these bounds become weaker than the trivial estimate $E_{2,2}(A) \leq |A|^3$. 
\par

In the latter situation, the second regime of results becomes more efficient, which consists of bounds for $E_{s,2}(A)$ that are independent of the spacing $\delta_{A}$. For instance, Bourgain and Demeter \cite{BD2015} showed that
\begin{equation} \label{bou11}
 E_{2,2}(A) \ll |A|^{2 + 1/3}, 
 \end{equation}
for any finite, arbitrary subset $A$ of $\mathcal{S}_{3,1}$. We recall that such additive energies are very closely related to restriction estimates. Furthermore, since the restriction theory for the paraboloid and the sphere are very similar, and since it is possible to utilise incidence geometric methods to achieve the estimate
\[ E_{2,2}(A) \ll_{\epsilon} |A|^{2 + \epsilon}, \]
for each finite subset $A$ of the truncated paraboloid (see \cite[Theorem $2.8$]{BD2015}), Bourgain and Demeter \cite{BD2015} conjectured the analogous upper bound for additive energies on the $2$-sphere.
 
\begin{Conjecture} \label{beso}
Let $A$ be a finite, non-empty subset of $\mathcal{S}_{3,1}$. Then 
\[ E_{2,2}(A) \ll_{\epsilon} |A|^{2 + \epsilon}. \]
\end{Conjecture}

As in the previous sections, we can use point--sphere incidences to study this problem, and so, we begin by recording such a result by Zahl \cite[Theorem $1.2$]{Za2013}.

\begin{lemma} \label{zain}
Let $P$ be a finite set of points in $\mathbb{R}^3$ and let $L$ be a finite set of spheres in $\mathbb{R}^3$ such that no three spheres intersect in a common circle. Then
\[ \sum_{p \in P} \sum_{l \in L} \mathds{1}_{p \in l} \ll |P|^{3/4} |L|^{3/4} + |P| + |L|. \]
\end{lemma}

For our purposes, we will set $L = \{ \vec{x} + \mathcal{S}_{3,1} \ | \ \vec{x} \in X\}$ for some finite, non-empty set $X$. With this specific description of $L$ in hand, we can see that any three distinct spheres from $L$ can intersect in at most $O(1)$ elements in $\mathbb{R}^3$, and so, the hypothesis of Lemma $\ref{zain}$ is satisfied. As in the previous sections, we can then combine the conclusion of Lemma $\ref{zain}$ along with Lemma $\ref{6on}$ to attain a weighted incidence bound, which we can subsequently utilise to obtain upper bounds for $E_{s,2}(A)$ in terms of $E_{s-1,2}(A)$, whenever $A$ is a finite, non-empty subset of $\mathcal{S}_{3,1}$. In particular, we will get an inequality of the shape
\begin{equation} \label{ste2}
 E_{s,2}(A) \ll E_{s-1,2}(A)^{1/3} |A|^{(4s-2)/3} + |A|^{2s-2}, 
 \end{equation}
whenever $s \geq 2$. Setting $s=2$ in the above expression recovers the bound $\eqref{bou11}$, that is, the aforementioned result of Bourgain--Demeter. We further remark that upon adapting our argument from \S7 in this situation, we would be able to improve upon $\eqref{bou11}$ and obtain the estimate
\[ E_{2,2}(A) \ll |A|^{2 + 1/3  - 1/1030}, \]
for every finite, non-empty subset $A$ of $\mathcal{S}_{3,1}$. This can then be amalgamated with $\eqref{ste2}$ to obtain the threshold breaking bounds
\[ E_{s,2}(A) \ll |A|^{2s -2 + (1- 3/1030) \cdot 3^{-s+1}}, \]
whenever $s \geq 2$. 
\par

It was noted by Sheffer in \cite{Shef} that one should be able to obtain stronger bounds than $\eqref{bou11}$ by utilising point--circle incidences in $\mathbb{R}^3$. In order to see this, we first record the precise point--circle incidence result which we intend to employ (see \cite{AKS2005}). 

\begin{lemma} \label{aksi}
Let $P$ be a finite set of points in $\mathbb{R}^3$ and let $L$ be a finite collection of distinct circles in $\mathbb{R}^3$. Then 
\[ \sum_{p \in P} \sum_{l \in L} \mathds{1}_{p \in l} \ll_{\epsilon} |P|^{\epsilon} (|P|^{6/11} |L|^{9/11} + |P|^{2/3} |L|^{2/3} + |P| + |L|). \]
\end{lemma}

As before, we can then combine this with Lemma $\ref{6on}$ to attain a weighted incidence bound between points and circles in $\mathbb{R}^3$.

\begin{lemma} \label{wtaks}
Let $P$ be a finite set of points in $\mathbb{R}^3$ and let $L$ be a finite collection of distinct circles in $\mathbb{R}^3$ and let $w: L \to \mathbb{N}$ be a weight function. Then 
\[ \sum_{p \in P} \sum_{l \in L} \mathds{1}_{p \in l} \ll_{\epsilon} |P|^{\epsilon}(|P|^{6/11} \n{w}_2^{4/11}\n{w}_1^{7/11} + |P|^{2/3} \n{w}_2^{2/3} \n{w}_1^{1/3} + \n{w}_1 +  |P|\n{w}_{\infty}). \]
\end{lemma}

With the relevant incidence result in hand, we begin our analysis of $E_{2,2}(A)$. Thus, we write
\[ E_{2,2}(A) = \sum_{\vec{a}_1, \dots, \vec{a}_4 \in A} \mathds{1}_{\vec{a}_1 + \vec{a}_2 - \vec{a}_3 = \vec{a}_4} = \sum_{\vec{n} \in 2A} \sum_{\vec{a}, \vec{b} \in A} r_{2}(\vec{n}) \mathds{1}_{\vec{n} - \vec{a} = \vec{b}}. \]
We may assume that $\vec{n} \neq 0$, since the case when $\vec{n} = 0$ contributes at most $O(|A|^2)$ solutions to $E_{2,2}(A)$. Moreover, note that for a fixed $\vec{n} \in 2A \setminus \{0\}$, the expression $\vec{n} - \vec{a} = \vec{b}$ implies that $\vec{b} \in A \cap (\vec{n}- A) \subseteq D_{\vec{n}}$, where $D_{\vec{n}}$ is the unique circle described by the set $\mathcal{S}_{3,1} \cap (\vec{n} - \mathcal{S}_{3,1})$. Furthermore, since fixing the values of $\vec{b}$ and $\vec{n}$ also fixes $\vec{a}$, we deduce that
\[ \sum_{\vec{a}, \vec{b} \in A} \mathds{1}_{\vec{n} - \vec{a} = \vec{b}} = \sum_{\vec{b} \in A} \mathds{1}_{\vec{b} \in A \cap (\vec{n}- A)} \leq \sum_{\vec{b} \in A} \mathds{1}_{\vec{b} \in D_{\vec{n}}}, \]
whence,
\[ E_{2,2}(A) \leq \sum_{\vec{n} \in 2A \setminus \{0\}} \sum_{\vec{b} \in A} r_{2}(\vec{n}) \mathds{1}_{\vec{b} \in D_{\vec{n}}} + O(|A|^2). \]
\par

We can estimate the sum on the right hand side above using Lemma $\ref{wtaks}$, and so, we get
\[ E_{2,2}(A)  \ll_{\epsilon} |A|^{\epsilon}(|A|^{6/11} \n{r_2}_2^{4/11}\n{r_2}_1^{7/11} + |A|^{2/3} \n{r_2}_2^{2/3} \n{r_2}_1^{1/3} + \n{r_2}_1 +  |A|\n{r_2}_{\infty}) + |A|^2. \]
As before, we see that $\n{r_2}_2^2 \leq E_{2,2}(A)$ and $\n{r_2}_1 \leq |A|^2$ and $\n{r_2}_{\infty} \leq |A|,$ and consequently, the preceding inequality gives us
\[ E_{2,2}(A) \ll_{\epsilon} |A|^{\epsilon} (|A|^{20/11} E_{2,2}(A)^{2/11} +  |A|^{4/3} E_{2,2}(A)^{1/3} + |A|^2). \]
Simplifying the above yields the bound
\[ E_{2,2}(A) \ll_{\epsilon} |A|^{2 + 2/9+ \epsilon}, \]
which can subsequently be combined with the inductive estimate $\eqref{ste2}$ to deliver Theorem $\ref{sdlyf}$, and so, we are done.


\bibliographystyle{amsbracket}

\begin{thebibliography}{18}



\bibitem{AKS2005}
B. Aronov, V. Koltun, M. Sharir, \emph{Incidences between points and circles in three and higher dimensions}, Discrete Comput. Geom. \textbf{33} (2005), no. 2, 185-206.

\bibitem{BM2019}
J. Benatar, R.W. Maffucci, \emph{Random waves on $\mathbb{T}^3$: nodal area variance and lattice point correlations}, Int. Math. Res. Not. IMRN \textbf{2019}, no. 10, 3032-3075.

\bibitem{BB2015}
E. Bombieri, J. Bourgain, \emph{A problem on sums of two squares}, Int. Math. Res. Not. IMRN \textbf{2015}, no. 11, 3343-3407.




\bibitem{Bo1993a}
J. Bourgain, \emph{Eigenfunction bounds for the Laplacian on the $n$-torus}, Internat. Math. Res. Notices \textbf{1993}, no. 3, 61-66.





\bibitem{BD2015a}
J. Bourgain, C. Demeter, \emph{New bounds for the discrete Fourier restriction to the sphere in 4D and 5D}, Int. Math. Res. Not. IMRN \textbf{2015}, no. 11, 3150-3184.

\bibitem{BD2015}
J. Bourgain, C. Demeter, \emph{The proof of the $l^2$ decoupling conjecture}, Ann. of Math. (2) \textbf{182} (2015), no. 1, 351-389.



\bibitem{BSR2016}
J. Bourgain, P. Sarnak, Z. Rudnick, \emph{Local statistics of lattice points on the sphere}, Modern trends in constructive function theory, 269-282, Contemp. Math., \textbf{661}, Amer. Math. Soc., Providence, RI, 2016.


\bibitem{De2014}
C. Demeter, \emph{Incidence theory and restriction estimates}, preprint available as arXiv:1401.1873.

\bibitem{FPS2017}
J. Fox, J. Pach, A. Sheffer,  A. Suk, J. Zahl, \emph{A semi-algebraic version of Zarankiewicz's problem.} J. Eur. Math. Soc. (JEMS) \textbf{19} (2017), no. 6, 1785-1810.

\bibitem{GG2019}
P. T. Gressman, S. Guo, L. B. Pierce, J. Roos, P.-L. Yung, \emph{Reversing a philosophy: from counting to square functions and decoupling}, preprint available as arXiv:1906.05877.

\bibitem{HW1979}
G. H. Hardy, E. M. Wright, \emph{An introduction to the theory of numbers}, Fifth edition. The Clarendon Press, Oxford University Press, New York, 1979. xvi+426 pp. 



\bibitem{Ak2020}
A. Mudgal, \emph{Arithmetic Combinatorics on Vinogradov systems}, Trans. Amer. Math. Soc. \textbf{373} (2020), no. 8, 5491-5516.


\bibitem{Ak2020b}
A. Mudgal, \emph{Diameter free estimates for the quadratic Vinogradov mean value theorem}, preprint available as arXiv:2008.09247.


%
\bibitem{SS2011}
T. Schoen, I. Shkredov, \emph{On sumsets of convex sets}, Combin. Probab. Comput. \textbf{20} (2011), no. 5, 793-798.
%


\bibitem{She2016}
A. Sheffer, \emph{Lower bounds for incidences with hypersurfaces}, Discrete Anal. (2016), Paper No. 16, 14 pp.


\bibitem{Shef}
A. Sheffer, \emph{Polynomial Methods and Incidence Theory}, draft of book available online.

\bibitem{Sh2013}
I. Shkredov, \emph{Some new results on higher energies}, Trans. Moscow Math. Soc. 2013, 31-63.



\bibitem{TV2006}
T. Tao, V. H. Vu, \emph{Additive combinatorics}, Cambridge Studies in Advanced Mathematics, \textbf{105}. Cambridge University Press, Cambridge, 2006.



\bibitem{Za2013}
J. Zahl, \emph{An improved bound on the number of point-surface incidences in three dimensions}, Contrib. Discrete Math. \textbf{8} (2013), no. 1, 100-121.

\end{thebibliography}
\providecommand{\bysame}{\leavevmode\hbox to3em{\hrulefill}\thinspace}

\end{document}